\documentclass[12pt,leqno]{article}
\usepackage{amsfonts}
\usepackage{amsmath, amsthm, amsfonts, amssymb, color, ulem}
\usepackage{mathrsfs}

\newtheorem{thm}{Theorem}[section]

\newtheorem{lem}[thm]{Lemma}

\theoremstyle{definition}

\newcommand{\scr}[1]{\mathscr #1}
\definecolor{wco}{rgb}{0.5,0.2,0.3}
\renewcommand{\bar}{\overline}

\numberwithin{equation}{section}
\newtheorem{rem}{Remark}[section]

\newcommand{\ua}{\uparrow}

\renewcommand{\hat}{\widehat}
\renewcommand{\tilde}{\widetilde}

\title{{\bf Strong convergence of a tamed theta scheme for NSDDEs with one-sided Lipschitz drift\thanks{Supported by NSFC(No., 11561027, 11661039), NSF of Jiangxi(No., 20161BAB211018), Scientific Research Fund of Jiangxi Provincial Education Department(No., GJJ150444).}}
}

\author
{ {\bf Li Tan}$^{\tt a, b}$ \, and \,
	{\bf Chenggui Yuan}$^{\tt c}$\thanks{Contact e-mail address: C.Yuan@swansea.ac.uk}\
	\\[0.5ex]
	$^{\tt a}$School of Statistics, Jiangxi University of Finance and Economics,\\
	Nanchang, Jiangxi, 330013, P. R. China \\
    $^{\tt b}$Research Center of Applied Statistics, Jiangxi University of Finance\\
	 and Economics, Nanchang, Jiangxi, 330013, P. R. China \\
	$^{\tt c}$Department of Mathematics\\
	Swansea University, Swansea,  SA2 8PP, U.K.\\}

\begin{document}
\def\R{\mathbb R}  \def\ff{\frac} \def\ss{\sqrt} \def\B{\mathbf
B}
\def\N{\mathbb N} \def\kk{\kappa} \def\m{{\bf m}}
\def\dd{\delta} \def\DD{\Delta} \def\vv{\varepsilon} \def\rr{\rho}
\def\<{\langle} \def\>{\rangle} \def\GG{\Gamma} \def\gg{\gamma}
  \def\nn{\nabla} \def\pp{\partial} \def\EE{\scr E}
\def\d{\text{\rm{d}}} \def\bb{\beta} \def\aa{\alpha} \def\D{\scr D}
  \def\si{\sigma} \def\ess{\text{\rm{ess}}}
\def\beg{\begin} \def\beq{\begin{equation}}  \def\F{\scr F}
\def\Ric{\text{\rm{Ric}}} \def\Hess{\text{\rm{Hess}}}
\def\e{\text{\rm{e}}} \def\ua{\underline a} \def\OO{\Omega}  \def\oo{\omega}
 \def\Ric{\text{\rm{Ric}}}
\def\cut{\text{\rm{cut}}} \def\P{\mathbb P} \def\ifn{I_n(f^{\bigotimes n})}
\def\C{\scr C}      \def\aaa{\mathbf{r}}     \def\r{r}
\def\gap{\text{\rm{gap}}} \def\prr{\pi_{{\bf m},\varrho}}  \def\r{\mathbf r}
\def\Z{\mathbb Z} \def\vrr{\varrho} \def\ll{\lambda}
\def\L{\scr L}\def\Tt{\tt} \def\TT{\tt}\def\II{\mathbb I}
\def\i{{\rm in}}\def\Sect{{\rm Sect}}\def\E{\mathbb E} \def\H{\mathbb H}
\def\M{\scr M}\def\Q{\mathbb Q} \def\texto{\text{o}} \def\LL{\Lambda}
\def\Rank{{\rm Rank}} \def\B{\scr B} \def\i{{\rm i}} \def\HR{\hat{\R}^d}
\def\to{\rightarrow}\def\l{\ell}
\def\8{\infty}\def\Y{\mathbb{Y}}

\maketitle

\begin{abstract}
This paper is concerned with strong convergence of a tamed theta scheme for neutral stochastic differential delay equations with one-sided Lipschitz drift. Strong convergence rate is revealed under a global one-sided Lipschitz condition, while for a local one-sided Lipschitz condition, the tamed theta scheme is modified to ensure the well-posedness of implicit numerical schemes, then we show the convergence of the numerical solutions.
 \end{abstract}
 \vskip 2mm

\section{Introduction}
Numerical analysis plays an important role in studying stochastic differential equations (SDEs) because most equations can not be solved explicitly. The most commonly used method for approximating SDEs is the explicit Euler-Maruyama (EM) method. There are a lot of literature concerning with the explicit EM scheme for all kinds of SDEs, e.g.,  Hairer et al. \cite{Hai16},  Maruyama \cite{m55}, Milstein \cite{m95}, and Kloeden and Platen \cite{kp92}. Most of the early works on explicit EM scheme were about the SDEs with the globally Lipschitz continuous coefficients, since the explicit EM scheme solutions may not converge in the strong sense to the exact solutions with one-sided Lipschitz continuous and superlinearly growing drift coefficients. Moreover,  Hutzenthaler et al. \cite{hjk11} pointed out  that the absolute moments of the EM scheme at a finite time could diverge to infinity.  In order to cope with these difficulties,  Higham et.al \cite{hms02} studied a split-step backward Euler method  for nonlinear SDEs, they showed that the implicit EM scheme converged if the drift coefficient satisfied a one-sided Lipschitz condition and the diffusion coefficient was globally Lipschitz.  Hutzenthaler et al. \cite{hjk12}
proposed a tamed EM scheme in which the drift term is modified to guarantee the boundness of moments. Later, Sabanis \cite{sabanis13, sabanis15} studied the strong convergence of the tamed EM scheme and extend the tamed EM scheme to SDEs with superlinearly growing drift and diffusion coefficients, respectively. Although additional computational effort is needed for implicit analysis, the implicit EM schemes have been  showed better than the explicit EM scheme which converges strongly to the exact solution of SDEs under non-globally Lipschitz conditions. The implicit EM methods including the backward EM scheme, the split-step backward EM scheme and the theta scheme have been extensively studied, for example, Mao and Szpruch \cite{ms13} studied strong convergence and almost sure stability of the backward EM scheme and the theta scheme to SDEs with non-linear and non-Lipschitzian coefficients, to name a few.

Recently, numerical analysis for neutral stochastic differential delay equations (NSDDEs) has also received a great deal of attention, see e.g., Lan and Yuan \cite{ly15}, Wu and Mao \cite{wm08}, Zhou \cite{z15}, Zong et al. \cite{zwh15}, Zong and Huang \cite{zw16}, and the references therein. However, the existing literature are difficult to deal with one-sided Lipschitz and superlinearly drift. To fill the gap, in this paper, we are going to introduce a tamed theta scheme and  discuss the strong convergence of this scheme for NSDDEs in which the drift coefficients are one-sided Lipschitz and superlinearly.

The content of our paper is organized as follows. In section 2, we consider NSDDEs with global one-sided Lipschitz drift, the tamed theta scheme is introduced and strong convergence is investigated. We reveal that the tamed theta solution converges to the exact solution with order $\alpha$ (see (B1) below)under the global one-sided Lipschitz and the superlinearly growth condition. In section 3, the global one-sided Lipschitiz drift is replaced by the local one-sided Lipschitz drift, under which we show the convergence of the numerical solutions. In order to guarantee the well-posedness of the implicit tamed scheme, we impose a modified tamed theta scheme with a truncated skill.


\section{\bf Global One-sided Lipschitz Drift}
For a fixed positive integer $n$, let
$(\mathbb{R}^n, \<\cdot,\cdot\>, |\cdot|)$ be an $n$-dimensional Euclidean
space. Denote $\mathbb{R}^n\otimes\mathbb{R}^d$ by the set of all $n\times d$
matrices  endowed with Hilbert-Schmidt norm
 $\|A\|:=\sqrt{\mbox{trace}(A^*A)}$ for every $A\in \mathbb{R}^n\otimes\mathbb{R}^d $, in which $A^*$ is the transpose of $A$. For a fixed $\tau\in(0,\infty)$, which will be referred
to as the delay or memory, let $\mathscr{C}=C([-\tau,0];\mathbb{R}^n)$ be all  continuous functions from $[-\tau,0]$ to $\mathbb{R}^n$, equipped
with the uniform norm
$\|\zeta\|_\infty:=\sup_{-\tau\le\theta\le 0}|\zeta(\theta)|$ for
every $\zeta\in\C$. By a filtered probability space, we mean a
quadruple $(\Omega, \mathscr{F}, \{\mathscr{F}_t\}_{t\ge 0}, \mathbb{P})$, where $\mathscr{F}$ is a
$\sigma$-algebra on the outcome space $\Omega$, $\mathbb{P}$ is a probability
measure on the measurable space $(\Omega,\mathscr{F})$, and $\{\mathscr{F}_t\}_{t\ge0}$ is
a filtration of sub-$\sigma$-algebra of $\mathscr{F}$, where the usual
conditions are satisfied, i.e., $(\Omega,\mathscr{F},\mathbb{P})$ is a complete
probability space, and  $\mathscr{F}_0$ contains all
$\mathbb{P}$-null sets of $\mathscr{F}$ and $\mathscr{F}_{t_+}:=\bigcap_{s>t}\mathscr{F}_s=\mathscr{F}_t.$
 Let $\{W(t)\}_{t\ge0}$ be a
$d$-dimensional Brownian motion defined on the filtered
 probability space $(\Omega,\mathscr{F},\{\mathscr{F}_t\}_{t\ge0}, \mathbb{P})$.

In this paper, we consider the following NSDDE
\begin{equation}\label{brownian}
\mbox{d}[X(t)-D(X(t-\tau))]=b(X(t), X(t-\tau))\mbox{d}t+\sigma(X(t), X(t-\tau))\mbox{d}W(t), t\ge 0
\end{equation}
with initial data
\begin{equation*}
X_0=\xi=\{\xi(\theta):-\tau\le\theta\le 0\}\in \mathcal{L}^p_{\mathscr{F}_0}([-\tau, 0]; \mathbb{R}^n), p\ge 2,
\end{equation*}
that is, $\xi$ is an $\mathscr{F}_0$-measurable $\mathscr{C}$-valued random variable such that $\mathbb{E}\|\xi\|^p_\infty<\infty$ for $p\ge 2$. Here, $D:\mathbb{R}^n\rightarrow\mathbb{R}^n$, and $b:\mathbb{R}^n\times\mathbb{R}^n\rightarrow\mathbb{R}^n$, $\sigma:\mathbb{R}^n\times\mathbb{R}^n\rightarrow\mathbb{R}^n\otimes\mathbb{R}^d$  are  continuous in $x$ and $y$. Fix $T>\tau>0$, assume that $T$ and $\tau$ are rational numbers, and the step size $\Delta\in (0,1)$ be fraction of $T$ and $\tau$, so that there exist two positive integers $M, m$ such that $\Delta=T/M=\tau/m$. Throughout the paper, we shall denote $C$ by a generic positive constant, whose value may change from line to line. Further, for any $x,y,\bar{x},\bar{y}\in\R^n$, we shall assume that:
\begin{enumerate}
\item[{\bf (A1)}]For any $s, t\in [-\tau, 0]$ and $q>0$, there exists a positive constant $K_1$ such that
\begin{equation*}
\mathbb{E}\|\xi(s)-\xi(t)\|_{\infty}^q\le K_1|s-t|^q.
\end{equation*}
\item[{\bf (A2)}]$D(0)=0$, and there exists a positive constant $\kappa\in (0,1/2)$ such that
\begin{equation*}
|D(x)-D(\bar{x})|\le\kappa|x-\bar{x}|.
\end{equation*}
\item[{\bf (A3)}]There exists a positive constant $K_2$ such that
\begin{equation*}
\< x-D(y), b(x, y)\>\vee\|\sigma(x, y)\|^2\le K_2(1+|x|^2+|y|^2).
\end{equation*}
\item[{\bf (A4)}]There exist positive constants $l$, $K_3$ and  $K_4$ such that for some $p\ge2$
\begin{equation*}
\begin{split}
&\quad 2\< x-D(y)-\bar{x}+D(\bar{y}), b(x, y)-b(\bar{x}, \bar{y})\>+(p-1)\|\sigma(x, y)-\sigma(\bar{x}, \bar{y})\|^2\\
&\le K_3(|x-\bar{x}|^2+|y-\bar{y}|^2),
\end{split}
\end{equation*}
and
\begin{equation*}
|b(x, y)-b(\bar{x}, \bar{y})|\le K_4(1+|x|^l+|\bar{x}|^l+|y|^l+|\bar{y}|^l)(|x-\bar{x}|+|y-\bar{y}|).
\end{equation*}
\end{enumerate}

\begin{rem}
{\rm Due to the existence of implicitness and the neutral term, scopes of $\Delta$ and $\kappa$ in assumption (A2) are given in order to guarantee rationality.}
 \end{rem}

\begin{rem}\label{remark2}
{\rm If $b(x, y)$ satisfies (A4), then,  for any $x,y\in\mathbb{R}^n$, we have
\begin{equation*}
\begin{split}
|b(x, y)|&\le|b(x,y)-b(0,0)|+|b(0,0)|\le K_4(1+|x|^l+|y|^l)(|x|+|y|)+|b(0,0)|\\
&\le C(1+|x|+|x|^{l+1}+|y|+|y|^{l+1}),
\end{split}
\end{equation*}
where $C=K_4\vee|b(0,0)|$. If the coefficients satisfy  (A2) and (A4), then one has
\begin{equation*}
\begin{split}
&\quad (p-1)\|\sigma(x, y)-\sigma(\bar{x}, \bar{y})\|^2\\
&\le K_3(|x-\bar{x}|^2+|y-\bar{y}|^2)+2|x-D(y)-\bar{x}+D(\bar{y})||b(x, y)-b(\bar{x}, \bar{y})|\\
&\le C(1+|x|^l+|\bar{x}|^l+|y|^l+|\bar{y}|^l)(|x-\bar{x}|^2+|y-\bar{y}|^2).
\end{split}
\end{equation*}
}
\end{rem}

\begin{rem}
{\rm There are many examples such that the assumptions can be  verified. For example, let
\begin{equation*}
\begin{split}
D(y)=-ay,\quad b(x,y)=x-x^3+ay-a^3y^3,\quad \sigma(x,y)=x+ay,
\end{split}
\end{equation*}
for $x,y\in\mathbb{R}$, where $a$ is a constant such that $|a|<1/2$. It is easy to check that assumptions (A2)-(A4) are satisfied.}
\end{rem}

\begin{lem}\label{exactexist}
{\rm Let (A1)-(A4) hold, the NSDDE \eqref{brownian} admits
a unique strong global solution $X(t), t\in[0, T]$, and
\begin{equation*}
\mathbb{E}\left(\sup\limits_{0\le t\le T}|X(t)|^p\right)\le C
\end{equation*}
for any $p\ge 2$. One can consult \cite{jiyuan16} for more details.
}
\end{lem}

\subsection{The Tamed Theta Scheme}
Now we introduce a tamed theta scheme for \eqref{brownian}. For $k=-m, \cdots, 0$, set $y_{t_k}=\xi(k\Delta)$; For $k=0, 1, \cdots, M-1$, we form
\begin{equation}\label{discrete1}
\begin{split}
y_{t_{k+1}}-D(y_{t_{k+1-m}})&=y_{t_k}-D(y_{t_{k-m}})+\theta b_\Delta(y_{t_{k+1}}, y_{t_{k+1-m}})\Delta\\
&\quad+(1-\theta) b_\Delta(y_{t_{k}}, y_{t_{k-m}})\Delta+\sigma_\Delta(y_{t_{k}}, y_{t_{k-m}})\Delta W_{t_k},
\end{split}
\end{equation}
where $t_k=k\Delta$, and $\Delta W_{t_k}=W(t_{k+1})-W(t_k)$. Here $b_\Delta:\mathbb{R}^n\times\mathbb{R}^n\rightarrow\mathbb{R}^n$ is a continuous function, and  $\sigma_\Delta:\mathbb{R}^n\times\mathbb{R}^n\rightarrow\mathbb{R}^n\otimes\mathbb{R}^d$ is a measurable function, $b_\Delta$ and $\sigma_\Delta$ satisfy some conditions given below. Besides, $\theta\in [0,1]$ is an additional parameter that allows us to control the implicitness of the numerical scheme.
Since it is convenient to work with a continuous extension of a numerical method, we now define the equivalent continuous form for \eqref{discrete1}. Let $Y_{\Delta}(t)=\xi(t), t\in[-\tau,0]$. For $t\in[0,T]$,  we   define the corresponding continuous-time tamed theta scheme by
\begin{equation*}
\begin{split}
Y_{\Delta}(t)=&D(\bar{Y}_{\Delta}(t-\tau))+\xi(0)-D(\xi(-\tau))+\theta\int_0^{t} b_\Delta(\bar{Y}_{\Delta+}(s), \bar{Y}_{\Delta+}(s-\tau))\mbox{d}s\\
&+(1-\theta)\int_0^{t} b_\Delta(\bar{Y}_{\Delta}(s), \bar{Y}_{\Delta}(s-\tau))\mbox{d}s+\int_0^{t} \sigma_\Delta(\bar{Y}_{\Delta}(s), \bar{Y}_{\Delta}(s-\tau))\mbox{d}W(s),
\end{split}
\end{equation*}
here $\bar{Y}_{\Delta}(t)$ is defined by
\begin{equation}\label{barytdef}
\bar{Y}_{\Delta}(t)=y_{t_k}  \mbox{ and } \bar{Y}_{\Delta+}(t)=y_{t_{k+1}} \quad \mbox{for} \quad t \in[t_k, t_{k+1}) ,
\end{equation}
thus $\bar{Y}_{\Delta}(t-\tau)=y_{t_{k-m}}$, and $\bar{Y}_{\Delta+}(t-\tau)=y_{t_{k+1-m}}$. However, this $Y_{\Delta}(t)$ is not $\mathscr{F}_t$-adapted, it does not meet the fundamental requirement in the It\^o stochastic analysis. To avoid Malliavin calculus, we use the discrete split-step theta scheme introduced by Zong et al. \cite{zwh15} as follows: For $k=-m, \cdots, -1$, set $z_{t_k}=\xi(k\Delta)$.  For $k=0, 1, \cdots, M-1$, we reformulate  the scheme \eqref{discrete1} as follows
\begin{equation}\label{discrete2}
\begin{cases}
y_{t_{k}}=D(y_{t_{{k}-m}})+z_{t_{k}}-D(z_{t_{{k}-m}})+\theta b_\Delta(y_{t_{k}},y_{t_{{k}-m}})\Delta,\\
z_{t_{k+1}}=D(z_{t_{k+1-m}})+z_{t_k}-D(z_{t_{k-m}})+b_\Delta(y_{t_k},y_{t_{k-m}})\Delta+\sigma_\Delta(y_{t_k},y_{t_{k-m}})\Delta W_{t_k}.
\end{cases}
\end{equation}
 This scheme  can also be rewritten as
\begin{equation*}
\begin{split}
& z_{t_{k+1}}-D(z_{t_{k+1-m}})=z_{t_0}-D(z_{t_{-m}})+\sum\limits_{i=0}^kb_\Delta(y_{t_i},y_{t_{i-m}})\Delta+\sum\limits_{i=0}^k\sigma_\Delta(y_{t_i},y_{t_{i-m}})\Delta W_{t_i}\\
=&\xi(0)-D(\xi(-\tau))-\theta b_\Delta(\xi(0),\xi(-\tau))\Delta+\sum\limits_{i=0}^kb_\Delta(
y_{t_i},y_{t_{i-m}})\Delta+\sum\limits_{i=0}^k\sigma_\Delta(y_{t_i},y_{t_{i-m}})\Delta W_{t_i}.
\end{split}
\end{equation*}
In order to simplify the computation, we define the corresponding continuous-time split-step tamed theta solution $Z_\Delta(t)$ as follows: For any $t\in[-\tau,0)$, $Z_\Delta(t)=\xi(t)$, $Z_\Delta(0)=\xi(0)-\theta b_\Delta(\xi(0), \xi(-\tau))\Delta$.  For any $t\in[0,T]$,
\begin{equation}\label{continuous}
\mbox{d}[Z_\Delta(t)-D(Z_\Delta(t-\tau))]=b_\Delta(\bar{Y}_{\Delta}(t), \bar{Y}_{\Delta}(t-\tau))\mbox{d}t+\sigma_\Delta(\bar{Y}_{\Delta}(t), \bar{Y}_{\Delta}(t-\tau))\mbox{d}W(t),
\end{equation}
where $\bar{Y}_{\Delta}(t)$ is defined by \eqref{barytdef}. With the split-step tamed theta scheme \eqref{discrete2}, the continuous form of the split-step tamed theta solution $Z_\Delta(t)$ and the tamed theta solution $Y_{\Delta}(t)$ have the following relation:
\begin{equation*}
Y_{\Delta}(t)-D(Y_{\Delta}(t-\tau))-\theta b_\Delta(Y_{\Delta}(t),Y_{\Delta}(t-\tau))\Delta=Z_\Delta(t)-D(Z_\Delta(t-\tau)).
\end{equation*}
Denote $\tilde{Y}_{\Delta}(t)=Y_{\Delta}(t)-D(Y_{\Delta}(t-\tau))-\theta b_\Delta(Y_{\Delta}(t),Y_{\Delta}(t-\tau))\Delta$, we can rewrite \eqref{continuous} as
\begin{equation}\label{ycontinuous}
\tilde{Y}_{\Delta}(t)=\tilde{Y}_{\Delta}(0)+\int_0^tb_\Delta(\bar{Y}_{\Delta}(s), \bar{Y}_{\Delta}(s-\tau))\mbox{d}s+\int_0^t\sigma_\Delta(\bar{Y}_{\Delta}(s), \bar{Y}_{\Delta}(s-\tau))\mbox{d}W(s), t\in[0,T],
\end{equation}
where $\tilde{Y}_{\Delta}(0)=\xi(0)-D(\xi(-\tau))-\theta b_\Delta(\xi(0),\xi(-\tau))\Delta$. It is easy to see $\tilde{Y}_{\Delta}(t)$ coincides with $\bar{Y}_{\Delta}(t)-D(\bar{Y}_{\Delta}(t-\tau))-\theta b_\Delta(\bar{Y}_{\Delta}(t),\bar{Y}_{\Delta}(t-\tau))\Delta$ at grid points $t=k\Delta, k=0,1,\cdots,M-1$, this also means that the continuous-time tamed theta solution $Y_{\Delta}(t)$ coincides with the discrete-time tamed theta solution $\bar{Y}_{\Delta}(t)$ at grid points $t=k\Delta, k=0,1,\cdots,M-1$.

We need some assumptions on $b_\Delta(x,y)$ and $ \sigma_\Delta(x,y)$. We assume that there exists an $\alpha\in(0, 1/2]$ such that for any $x, y, \bar{x}, \bar{y} \in\R^n$, the following conditions hold:
\begin{enumerate}
\item[{\bf (B1)}]There exist a positive constant $K_5\ge 1$ such that
\begin{equation*}
|b_\Delta(x,y)|\le \min(K_5\Delta^{-\alpha}(1+|x|+|y|), |b(x,y)|),
\end{equation*}
and
\begin{equation*}
\|\sigma_\Delta(x,y)\|^2\le \min(K_5\Delta^{-\alpha}(1+|x|^2+|y|^2), \|\sigma(x,y)\|^2).
\end{equation*}
\item[{\bf (B2)}]There exists a positive constant $\tilde{K}_2$ such that
\begin{equation*}
\< x-D(y), b_\Delta(x,y)\>\vee\|\sigma_\Delta(x,y)\|^2\le \tilde{K}_2(1+|x|^2+|y|^2).
\end{equation*}
\item[{\bf (B3)}]There exists a positive constant $\tilde{K}_3$ such that
\begin{equation*}
\begin{split}
&\quad \< x-D(y)-\bar{x}+D(\bar{y}), b_\Delta(x, y)-b_\Delta(\bar{x}, \bar{y})\>\le \tilde{K}_3(|x-\bar{x}|^2+|y-\bar{y}|^2).
\end{split}
\end{equation*}
\item[{\bf (B4)}]There exist positive constants $l$, $K_6$ such that for $p\ge2$
\begin{equation*}
\begin{split}
|b(x,y)-b_\Delta(x,y)|^{p}\vee\|\sigma(x,y)-\sigma_\Delta(x,y)\|^{p}\le K_6\Delta^{\alpha p}[1+|x|^{(2l+1)p}+|y|^{(2l+1)p}].
\end{split}
\end{equation*}
\end{enumerate}

\begin{rem}
{\rm With assumptions (B1)-(B3), \eqref{discrete1} is well defined under some constraints on time step $\Delta$. It's worth pointing out that the assumption (B3) is merely used to guarantee the uniqueness of numerical solutions. In fact, (B2) can be derived from (A2), (A3),  (B1) and (B3), however, since (B2) will be used  in the proof  and can be replaced by a weaker form (see Remark \ref{weaker} for more details) while (B3) is not needed in moment estimation and strong convergence, we impose (B2) there.
}
\end{rem}
In order to ensure the implicitness of scheme \eqref{discrete1} is well defined, an additional restriction is required on time step, i.e. $\theta\Delta\tilde{K}_3<1$, where $\tilde{K}_3$ is defined in (B3) (see \cite{zei} for more details). For $\theta\in(0,1]$, denote $\Delta_1=\frac{1}{\theta\tilde{K}_3}$. Further, in order to guarantee the boundedness of the $p$-th moment of numerical solutions, the step size is also required  to satisfy $\theta^p\Delta<6^{1-p}(2^{-p}-\kappa^p)/K_5^p$ for $p\ge 2$ where $\kappa$ and $K_5$ are defined  in (A2) and (B1). Denote $\Delta_2=6^{1-p}(2^{-p}-\kappa^p)/(\theta^p K_5^p)$ for $\theta\in(0,1]$. Thus in this section, we set $\Delta^*\in(0,\Delta_1\wedge\Delta_2)$, and let $0<\Delta\le\Delta^*$ for $\theta\in(0,1]$, while for $\theta=0$, we may set $\Delta\in(0,1)$.

\begin{rem}
{\rm Under conditions (A2)-(A4), the set of sequences of functions which satisfy (B1)-(B4) are non-empty. For example, let  $b(x, y), \sigma(x, y): \mathbb{R}\times\mathbb{R} \rightarrow \mathbb{R}$, define
\begin{equation*}
b_\Delta(x,y)=\frac{b(x,y)}{1+\Delta^{\alpha}|b(x,y)|},
\end{equation*}
and
\begin{equation*}
\sigma_\Delta(x,y)=\frac{\sigma(x,y)}{1+\Delta^{\alpha}\|\sigma(x,y)\|^2}.
\end{equation*}
It is easy to see $|b_\Delta(x, y)|\le |b(x,y)|$, and on the other hand, we have
\begin{equation*}
|b_\Delta(x,y)|=\Delta^{-\alpha}\frac{|b(x,y)|}{\Delta^{-\alpha}+|b(x,y)|}\le \Delta^{-\alpha}\le K_5\Delta^{-\alpha}(1+|x|+|y|),
\end{equation*}
and
\begin{equation*}
\|\sigma_\Delta(x,y)\|^2=\Delta^{-\alpha}\frac{\Delta^{-\alpha}\|\sigma(x,y)\|^2}{(\Delta^{-\alpha}+\|\sigma(x,y)\|^2)^2}\le \Delta^{-\alpha}\le K_5\Delta^{-\alpha}(1+|x|^2+|y|^2).
\end{equation*}
That is, (B1) is verified. Furthermore, due to (A3), we have
\begin{equation*}
\begin{split}
&\quad \< x-D(y), b_\Delta(x,y)\>=\frac{1}{1+\Delta^{\alpha}|b(x,y)|}\< x-D(y), b(x,y)\>\\
&\le K_2(1+|x|^2+|y|^2),
\end{split}
\end{equation*}
and
\begin{equation*}
\begin{split}
&\|\sigma_\Delta(x,y)\|^2\le\|\sigma(x,y)\|^2\le K_2(1+|x|^2+|y|^2).
\end{split}
\end{equation*}
Then, we see (B2) holds. We are now going to check (B3), we divide it into two cases. For $b(x,y)\cdot b(\bar{x},\bar{y})<0$,
\begin{equation*}
\begin{split}
&\quad \< x-D(y)-\bar{x}+D(\bar{y}), b_\Delta(x,y)-b_\Delta(\bar{x},\bar{y})\>\\
=&\left\< x-D(y)-\bar{x}+D(\bar{y}), \frac{b(x,y)}{1+\Delta^{\alpha}|b(x,y)|}-\frac{b(\bar{x},\bar{y})}{1+\Delta^{\alpha}|b(\bar{x},\bar{y})|}\right\>\\
\le& \frac{1}{2}K_3(|x-\bar{x}|^2+|y-\bar{y}|^2).
\end{split}
\end{equation*}
For $b(x,y)\cdot b(\bar{x},\bar{y})>0$,
\begin{equation*}
\begin{split}
&\quad \< x-D(y)-\bar{x}+D(\bar{y}), b_\Delta(x,y)-b_\Delta(\bar{x},\bar{y})\>\\
=&\left\< x-D(y)-\bar{x}+D(\bar{y}), \frac{b(x,y)-b(\bar{x},\bar{y})}{(1+\Delta^{\alpha}|b(x,y)|)(1+\Delta^{\alpha}|b(\bar{x},\bar{y})|)}\right\>\\
&+\left\< x-D(y)-\bar{x}+D(\bar{y}), \frac{\Delta^\alpha [b(x,y)|b(\bar{x},\bar{y})|-|b(x,y)|b(\bar{x},\bar{y})]}{(1+\Delta^{\alpha}|b(x,y)|)(1+\Delta^{\alpha}|b(\bar{x},\bar{y})|)}\right\>\\
\le& \frac{1}{2}K_3(|x-\bar{x}|^2+|y-\bar{y}|^2).
\end{split}
\end{equation*}
Due to Remark \ref{remark2}, we see that
\begin{equation*}
|b(x,y)-b_\Delta(x,y)|^{p}\le \Delta^{\alpha p}|b(x,y)|^{2p}\le K_6\Delta^{\alpha p}\left[1+|x|^{2(l+1)p}+|y|^{2(l+1)p}\right].
\end{equation*}
Similarly, we can also verify that $\sigma(x,y), \sigma_\Delta(x,y)$ satify (B4).
}
\end{rem}

\subsection{\bf Moment Bounds}

In order to prove the main results,   we now give some estimates for the numerical solution $Y_{\Delta}(t).$

\begin{lem}\label{pmoment}
{\rm
Let (A1)-(A2) and (B1)-(B2) hold. Then  it  holds that for any $p\ge 2$,
\begin{equation*}\label{ytpmoment}
\sup\limits_{0\le t \le T}\mathbb{E}|Y_{\Delta}(t)|^p\le C,
\end{equation*}
where the positive constant $C$ is independent of $\Delta$.
}
\end{lem}
\begin{proof}
For $a>0$, let $\lfloor a\rfloor$ be the integer part of $a$. Applying the It\^{o} formula to $[1+|\tilde{Y}_{\Delta}(t)|^{2}]^{\frac{p}{2}}$, we obtain
\begin{equation*}
\begin{split}
\mathbb{E}[1+|\tilde{Y}_{\Delta}(t)|^{2}]^{\frac{p}{2}}\le&\mathbb{E}[1+|\tilde{Y}_{\Delta}(0)|^{2}]^{\frac{p}{2}}+\frac{p}{2}\mathbb{E}\int_0^t[1+|\tilde{Y}_{\Delta}(s)|^{2}]^{\frac{p-2}{2}}2\langle \tilde{Y}_{\Delta}(s), b_\Delta(\bar{Y}_{\Delta}(s),\bar{Y}_{\Delta}(s-\tau))\rangle\mbox{d}s\\
&+\frac{1}{2}p(p-1)\mathbb{E}\int_0^t[1+|\tilde{Y}_{\Delta}(s)|^{2}]^{\frac{p-2}{2}}\|\sigma_\Delta(\bar{Y}_{\Delta}(s),\bar{Y}_{\Delta}(s-\tau))\|^2\mbox{d}s\\
\le&\mathbb{E}[1+|\tilde{Y}_{\Delta}(0)|^{2}]^{\frac{p}{2}}+\frac{p}{2}\mathbb{E}\int_0^t[1+|\tilde{Y}_{\Delta}(s)|^{2}]^{\frac{p-2}{2}}(p-1)\|\sigma_\Delta(\bar{Y}_{\Delta}(s),\bar{Y}_{\Delta}(s-\tau))\|^2\mbox{d}s\\
&+\frac{p}{2}\mathbb{E}\int_0^t[1+|\tilde{Y}_{\Delta}(s)|^{2}]^{\frac{p-2}{2}}2\langle\bar{Y}_{\Delta}(s)-D(\bar{Y}_{\Delta}(s-\tau)), b_\Delta(\bar{Y}_{\Delta}(s),\bar{Y}_{\Delta}(s-\tau))\rangle\mbox{d}s\\
&+p\mathbb{E}\int_0^t[1+|\tilde{Y}_{\Delta}(s)|^{2}]^{\frac{p-2}{2}}\langle \tilde{Y}_{\Delta}(s)-\vec{Y}_\Delta(s), b_\Delta(\bar{Y}_{\Delta}(s),\bar{Y}_{\Delta}(s-\tau))\rangle\mbox{d}s\\
=: &\mathbb{E}[1+|\tilde{Y}_{\Delta}(0)|^{2}]^{\frac{p}{2}}+E_1(t)+E_2(t)+E_3(t),
\end{split}
\end{equation*}
where $\vec{Y}_{\Delta}(t)=\bar{Y}_{\Delta}(t)-D(\bar{Y}_{\Delta}(t-\tau))-\theta b_\Delta(\bar{Y}_{\Delta}(t),\bar{Y}_{\Delta}(t-\tau))\Delta$. With conditions (A2), (B1)-(B2), we have
\begin{equation*}
\begin{split}
E_1(t)+E_2(t)\le &C\mathbb{E}\int_0^t[1+|\tilde{Y}_{\Delta}(s)|^{2}]^{\frac{p-2}{2}}(1+|\bar{Y}_{\Delta}(s)|^2+|\bar{Y}_{\Delta}(s-\tau)|^2)\mbox{d}s\\
\le &C \mathbb{E}\int_0^t[[1+|\tilde{Y}_{\Delta}(s)|^{2}]^{\frac{p}{2}}+|\bar{Y}_{\Delta}(s)|^p+|\bar{Y}_{\Delta}(s-\tau)|^p]\mbox{d}s\\
\le &C \mathbb{E}\int_0^t\big[|Y_{\Delta}(s)|^p+|Y_{\Delta}(s-\tau)|^p+|\theta b_\Delta(Y_{\Delta}(s),Y_{\Delta}(s-\tau))\Delta|^p\\
&~~~~~~~~~~~~~~~~~~~~~~~~~~~~+|\bar{Y}_{\Delta}(s)|^p+|\bar{Y}_{\Delta}(s-\tau)|^p\big]\mbox{d}s\\
\le&C \mathbb{E}\int_0^t(|Y_{\Delta}(s)|^{p}+|Y_{\Delta}(s-\tau)|^p+|\bar{Y}_{\Delta}(s)|^p+|\bar{Y}_{\Delta}(s-\tau)|^p)\mbox{d}s\\
&+C\Delta^{(1-\alpha)p}\mathbb{E}\int_0^t(1+|Y_{\Delta}(s)|^{p}+|Y_{\Delta}(s-\tau)|^p)\mbox{d}s\\
\le&C+C\int_0^t\sup\limits_{0\le u\le s}\mathbb{E}|Y_{\Delta}(u)|^p\mbox{d}s.
\end{split}
\end{equation*}
Furthermore, it is easy to observe that,
\begin{equation*}
\begin{split}
&E_3(t)=p\mathbb{E}\int_0^t[1+|\vec{Y}_\Delta(s)|^{2}]^{\frac{p-2}{2}}\langle \tilde{Y}_{\Delta}(s)-\vec{Y}_\Delta(s), b_\Delta(\bar{Y}_{\Delta}(s),\bar{Y}_{\Delta}(s-\tau))\rangle\mbox{d}s\\
&+p\mathbb{E}\int_0^t\left\{[1+|\tilde{Y}_{\Delta}(s)|^{2}]^{\frac{p-2}{2}}-[1+|\vec{Y}_\Delta(s)|^{2}]^{\frac{p-2}{2}}\right\}\langle \tilde{Y}_{\Delta}(s)-\vec{Y}_\Delta(s), b_\Delta(\bar{Y}_{\Delta}(s),\bar{Y}_{\Delta}(s-\tau))\rangle\mbox{d}s\\
&=:pE_{31}(t)+pE_{32}(t),
\end{split}
\end{equation*}
where
\begin{equation*}
\tilde{Y}_{\Delta}(s)-\vec{Y}_\Delta(s)=\int_{\lfloor\frac{s}{\Delta}\rfloor\Delta}^sb_\Delta(\bar{Y}_{\Delta}(u),\bar{Y}_{\Delta}(u-\tau))\mbox{d}u+\int_{\lfloor\frac{s}{\Delta}\rfloor\Delta}^s\sigma_\Delta(\bar{Y}_{\Delta}(u),\bar{Y}_{\Delta}(u-\tau))\mbox{d}W(u).
\end{equation*}
Due to (B1) and the Young inequality,
\begin{equation*}
\begin{split}
&E_{31}(t)=\mathbb{E}\int_0^t[1+|\vec{Y}_\Delta(s)|^{2}]^{\frac{p-2}{2}}\bigg\langle\int_{\lfloor\frac{s}{\Delta}\rfloor\Delta}^sb_\Delta(\bar{Y}_{\Delta}(u),\bar{Y}_{\Delta}(u-\tau))\mbox{d}u,b_\Delta(\bar{Y}_{\Delta}(s),\bar{Y}_{\Delta}(s-\tau))\bigg\rangle\mbox{d}s\\
&+\mathbb{E}\int_0^t[1+|\vec{Y}_\Delta(s)|^{2}]^{\frac{p-2}{2}}\bigg\langle\mathbb{E}\int_{\lfloor\frac{s}{\Delta}\rfloor\Delta}^s\sigma_\Delta(\bar{Y}_{\Delta}(u),\bar{Y}_{\Delta}(u-\tau))\mbox{d}W(u)\bigg|_{\mathscr{F}_{\lfloor\frac{s}{\Delta}\rfloor\Delta}},b_\Delta(\bar{Y}_{\Delta}(s),\bar{Y}_{\Delta}(s-\tau))\bigg\rangle\mbox{d}s\\
&\le\mathbb{E}\int_0^t[1+|\vec{Y}_\Delta(s)|^{2}]^{\frac{p-2}{2}}\int_{\lfloor\frac{s}{\Delta}\rfloor\Delta}^s|b_\Delta(\bar{Y}_{\Delta}(u),\bar{Y}_{\Delta}(u-\tau))|\mbox{d}u|b_\Delta(\bar{Y}_{\Delta}(s),\bar{Y}_{\Delta}(s-\tau))|\mbox{d}s\\
&\le \Delta\mathbb{E}\int_0^t[1+|\vec{Y}_\Delta(s)|^{2}]^{\frac{p-2}{2}}|b_\Delta(\bar{Y}_{\Delta}(s),\bar{Y}_{\Delta}(s-\tau))|^2\mbox{d}s\\
&\le C\Delta^{1-2\alpha}\mathbb{E}\int_0^t(1+|\bar{Y}_{\Delta}(s)|^p+|\bar{Y}_{\Delta}(s-\tau)|^p)\mbox{d}s\\
&+C\Delta^{1-2\alpha}\Delta^{(1-\alpha)p}\mathbb{E}\int_0^t(1+|\bar{Y}_{\Delta}(s)|^{p}+|\bar{Y}_{\Delta}(s-\tau)|^p)\mbox{d}s\\
&\le C+C\int_0^t\sup\limits_{0\le u\le s}\mathbb{E}|Y_{\Delta}(u)|^p\mbox{d}s.
\end{split}
\end{equation*}
Applying the It\^{o} formula again, we obtain
\begin{equation*}
\begin{split}
&[1+|\tilde{Y}_{\Delta}(s)|^{2}]^{\frac{p-2}{2}}\\
&\le[1+|\tilde{Y}_{\Delta}(0)|^{2}]^{\frac{p-2}{2}}+(p-2)\int_0^s[1+|\tilde{Y}_{\Delta}(u)|^{2}]^{\frac{p-4}{2}}\langle \tilde{Y}_{\Delta}(u), b_\Delta(\bar{Y}_{\Delta}(u),\bar{Y}_{\Delta}(u-\tau))\rangle\mbox{d}u\\
&+\frac{1}{2}(p-2)(p-3)\int_0^s[1+|\tilde{Y}_{\Delta}(u)|^{2}]^{\frac{p-4}{2}}\|\sigma_\Delta(\bar{Y}_{\Delta}(u),\bar{Y}_{\Delta}(u-\tau))\|^2\mbox{d}u\\
&+(p-2)\int_0^s[1+|\tilde{Y}_{\Delta}(u)|^{2}]^{\frac{p-4}{2}}\langle \tilde{Y}_{\Delta}(u), \sigma_\Delta(\bar{Y}_{\Delta}(u),\bar{Y}_{\Delta}(u-\tau))\mbox{d}W(u)\rangle.
\end{split}
\end{equation*}
Thus we have
\begin{equation*}
\begin{split}
&[1+|\vec{Y}_\Delta(s)|^{2}]^{\frac{p-2}{2}}\\
&\le[1+|\vec{Y}_{\Delta}(0)|^{2}]^{\frac{p-2}{2}}+(p-2)\int_0^{\lfloor\frac{s}{\Delta}\rfloor\Delta}[1+|\tilde{Y}_{\Delta}(u)|^{2}]^{\frac{p-4}{2}}\langle \tilde{Y}_{\Delta}(u), b_\Delta(\bar{Y}_{\Delta}(u),\bar{Y}_{\Delta}(u-\tau))\rangle\mbox{d}u\\
&+\frac{1}{2}(p-2)(p-3)\int_0^{\lfloor\frac{s}{\Delta}\rfloor\Delta}[1+|\tilde{Y}_{\Delta}(u)|^{2}]^{\frac{p-4}{2}}\|\sigma_\Delta(\bar{Y}_{\Delta}(u),\bar{Y}_{\Delta}(u-\tau))\|^2\mbox{d}u\\
&+(p-2)\int_0^{\lfloor\frac{s}{\Delta}\rfloor\Delta}[1+|\tilde{Y}_{\Delta}(u)|^{2}]^{\frac{p-4}{2}}\langle \tilde{Y}_{\Delta}(u), \sigma_\Delta(\bar{Y}_{\Delta}(u),\bar{Y}_{\Delta}(u-\tau))\mbox{d}W(u)\rangle.
\end{split}
\end{equation*}
Hence,
\begin{equation*}
\begin{split}
E_{32}(t)\le&(p-2)\mathbb{E}\int_0^t\int_{\lfloor\frac{s}{\Delta}\rfloor\Delta}^s[1+|\tilde{Y}_{\Delta}(u)|^{2}]^{\frac{p-4}{2}}\langle \tilde{Y}_{\Delta}(u), b_\Delta(\bar{Y}_{\Delta}(u),\bar{Y}_{\Delta}(u-\tau))\rangle\mbox{d}u\\
\quad&\times\langle \tilde{Y}_{\Delta}(s)-\vec{Y}_\Delta(s),b_\Delta(\bar{Y}_{\Delta}(s),\bar{Y}_{\Delta}(s-\tau))\rangle\mbox{d}s\\
&+\frac{1}{2}(p-2)(p-3)\mathbb{E}\int_0^t\int_{\lfloor\frac{s}{\Delta}\rfloor\Delta}^s[1+|\tilde{Y}_{\Delta}(u)|^{2}]^{\frac{p-4}{2}}\|\sigma_\Delta(\bar{Y}_{\Delta}(u),\bar{Y}_{\Delta}(u-\tau))\|^2\mbox{d}u\\
\quad&\times\langle \tilde{Y}_{\Delta}(s)-\vec{Y}_\Delta(s),b_\Delta(\bar{Y}_{\Delta}(s),\bar{Y}_{\Delta}(s-\tau))\rangle\mbox{d}s\\
&+(p-2)\mathbb{E}\int_0^t\int_{\lfloor\frac{s}{\Delta}\rfloor\Delta}^s[1+|\tilde{Y}_{\Delta}(u)|^{2}]^{\frac{p-4}{2}}\langle \tilde{Y}_{\Delta}(u), \sigma_\Delta(\bar{Y}_{\Delta}(u),\bar{Y}_{\Delta}(u-\tau))\mbox{d}W(u)\rangle\\
\quad&\times\langle \tilde{Y}_{\Delta}(s)-\vec{Y}_\Delta(s),b_\Delta(\bar{Y}_{\Delta}(s),\bar{Y}_{\Delta}(s-\tau))\rangle\mbox{d}s\\
=:&(p-2)E_{321}+\frac{1}{2}(p-2)(p-3)E_{322}+(p-2)E_{323}.
\end{split}
\end{equation*}
Using (B1), the Young inequality, the H\"{o}lder inequality and the Burkholder-Davis-Gundy (BDG) inequality, we compute
\begin{equation*}
\begin{split}
&E_{321}(t)\le\mathbb{E}\int_0^t\int_{\lfloor\frac{s}{\Delta}\rfloor\Delta}^s[1+|\tilde{Y}_{\Delta}(u)|^{2}]^{\frac{p-4}{2}}\langle \tilde{Y}_{\Delta}(u), b_\Delta(\bar{Y}_{\Delta}(u),\bar{Y}_{\Delta}(u-\tau))\rangle\mbox{d}u\\
&\times\bigg\langle\int_{\lfloor\frac{s}{\Delta}\rfloor\Delta}^sb_\Delta(\bar{Y}_{\Delta}(u),\bar{Y}_{\Delta}(u-\tau))\mbox{d}u,b_\Delta(\bar{Y}_{\Delta}(s),\bar{Y}_{\Delta}(s-\tau))\bigg\rangle\mbox{d}s\\
&+\mathbb{E}\int_0^t\int_{\lfloor\frac{s}{\Delta}\rfloor\Delta}^s[1+|\tilde{Y}_{\Delta}(u)|^{2}]^{\frac{p-4}{2}}\langle \tilde{Y}_{\Delta}(u), b_\Delta(\bar{Y}_{\Delta}(u),\bar{Y}_{\Delta}(u-\tau))\rangle\mbox{d}u\\
&\times\bigg\langle\int_{\lfloor\frac{s}{\Delta}\rfloor\Delta}^s\sigma_\Delta(\bar{Y}_{\Delta}(u),\bar{Y}_{\Delta}(u-\tau))\mbox{d}W(u),b_\Delta(\bar{Y}_{\Delta}(s),\bar{Y}_{\Delta}(s-\tau))\bigg\rangle\mbox{d}s\\
\le&\Delta\mathbb{E}\int_0^t\int_{\lfloor\frac{s}{\Delta}\rfloor\Delta}^s[1+|\tilde{Y}_{\Delta}(u)|^{2}]^{\frac{p-3}{2}}|b_\Delta(\bar{Y}_{\Delta}(s),\bar{Y}_{\Delta}(s-\tau))|^3\mbox{d}u\mbox{d}s\\
&+C\mathbb{E}\int_0^t\bigg[\bigg(\int_{\lfloor\frac{s}{\Delta}\rfloor\Delta}^s[1+|\tilde{Y}_{\Delta}(u)|^{2}]^{\frac{p-3}{2}}|b_\Delta(\bar{Y}_{\Delta}(u),\bar{Y}_{\Delta}(u-\tau))|\mbox{d}u |b_\Delta(\bar{Y}_{\Delta}(s),\bar{Y}_{\Delta}(s-\tau))|\bigg)^{\frac{p}{p-1}}\\
&+\left|\int_{\lfloor\frac{s}{\Delta}\rfloor\Delta}^s\sigma_\Delta(\bar{Y}_{\Delta}(u),\bar{Y}_{\Delta}(u-\tau))\mbox{d}W(u)\right|^{p}\bigg]\mbox{d}s\\
\le&C\Delta^{2-3\alpha}\mathbb{E}\int_0^t(|Y_{\Delta}(s)|^{p}+|Y_{\Delta}(s-\tau)|^{p}+|\bar{Y}_{\Delta}(s)|^{p}+|\bar{Y}_{\Delta}(s-\tau)|^{p})\mbox{d}s\\
&+C\Delta^{2-3\alpha}\Delta^{(1-\alpha)p}\mathbb{E}\int_0^t(|Y_{\Delta}(s)|^{p}+|Y_{\Delta}(s-\tau)|^{p})\mbox{d}s\\
&+C\mathbb{E}\int_0^t\left(\int_{\lfloor\frac{s}{\Delta}\rfloor\Delta}^s[1+|\tilde{Y}_{\Delta}(u)|^{2}]^{\frac{p-3}{2}}|b_\Delta(\bar{Y}_{\Delta}(s),\bar{Y}_{\Delta}(s-\tau))|^2\mbox{d}u\right)^{\frac{p}{p-1}}\mbox{d}s\\
&+C\mathbb{E}\int_0^t\left(\int_{\lfloor\frac{s}{\Delta}\rfloor\Delta}^s\|\sigma_\Delta(\bar{Y}_{\Delta}(u),\bar{Y}_{\Delta}(u-\tau))\|^2\mbox{d}u\right)^{\frac{p}{2}}\mbox{d}s\\
\le&C+C\int_0^t\sup\limits_{0\le u\le s}\mathbb{E}|Y_{\Delta}(u)|^{p}\mbox{d}s
+C\Delta^{(1-2\alpha)\frac{p}{2}}\int_0^t\sup\limits_{0\le u\le s}\mathbb{E}|Y_{\Delta}(u)|^{p}\mbox{d}s\\
\le&C+C\int_0^t\sup\limits_{0\le u\le s}\mathbb{E}|Y_{\Delta}(u)|^{p}\mbox{d}s.
\end{split}
\end{equation*}
Using the same techniques in the way to estimate  $E_{321}(t)$, we get
\begin{equation*}
E_{322}(t)\le C+C\int_0^t\sup\limits_{0\le u\le s}\mathbb{E}|Y_{\Delta}(u)|^{p}\mbox{d}s.
\end{equation*}
Furthermore, by (B1), we have
\begin{equation*}
\begin{split}
E_{323}(t)=&\mathbb{E}\int_0^t\int_{\lfloor\frac{s}{\Delta}\rfloor\Delta}^s[1+|\tilde{Y}_{\Delta}(u)|^{2}]^{\frac{p-4}{2}}\langle \tilde{Y}_{\Delta}(u),\sigma_\Delta(\bar{Y}_{\Delta}(u),\bar{Y}_{\Delta}(u-\tau))\mbox{d}W(u)\rangle \\
&\times\bigg\langle\int_{\lfloor\frac{s}{\Delta}\rfloor\Delta}^sb_\Delta(\bar{Y}_{\Delta}(u),\bar{Y}_{\Delta}(u-\tau))\mbox{d}u,
b_\Delta(\bar{Y}_{\Delta}(s),\bar{Y}_{\Delta}(s-\tau))\bigg\rangle\mbox{d}s\\
&+\mathbb{E}\int_0^t\int_{\lfloor\frac{s}{\Delta}\rfloor\Delta}^s[1+|\tilde{Y}_{\Delta}(u)|^{2}]^{\frac{p-4}{2}}\langle \tilde{Y}_{\Delta}(u),\sigma_\Delta(\bar{Y}_{\Delta}(u),\bar{Y}_{\Delta}(u-\tau))\mbox{d}W(u)\rangle \\
&\times\bigg\langle\int_{\lfloor\frac{s}{\Delta}\rfloor\Delta}^s\sigma_\Delta(\bar{Y}_{\Delta}(u),\bar{Y}_{\Delta}(u-\tau))\mbox{d}W(u),
b_\Delta(\bar{Y}_{\Delta}(s),\bar{Y}_{\Delta}(s-\tau))\bigg\rangle\mbox{d}s\\
=&\mathbb{E}\int_0^t\int_{\lfloor\frac{s}{\Delta}\rfloor\Delta}^s[1+|\tilde{Y}_{\Delta}(u)|^{2}]^{\frac{p-4}{2}}\langle \tilde{Y}_{\Delta}(u),\sigma_\Delta(\bar{Y}_{\Delta}(u),\bar{Y}_{\Delta}(u-\tau))\mbox{d}W(u)\rangle \\
&\times\bigg\langle\int_{\lfloor\frac{s}{\Delta}\rfloor\Delta}^s\sigma_\Delta(\bar{Y}_{\Delta}(u),\bar{Y}_{\Delta}(u-\tau))\mbox{d}W(u),
b_\Delta(\bar{Y}_{\Delta}(s),\bar{Y}_{\Delta}(s-\tau))\bigg\rangle\mbox{d}s\\
\le&\mathbb{E}\int_0^t\int_{\lfloor\frac{s}{\Delta}\rfloor\Delta}^s[1+|\tilde{Y}_{\Delta}(u)|^{2}]^{\frac{p-3}{2}}\|\sigma_\Delta(\bar{Y}_{\Delta}(u),\bar{Y}_{\Delta}(u-\tau))\|^2\mbox{d}u|b_\Delta(\bar{Y}_{\Delta}(s),\bar{Y}_{\Delta}(s-\tau))|\mbox{d}s\\
\le&C\Delta^{1-2\alpha}\mathbb{E}\int_0^t(|Y_{\Delta}(s)|^{p}+|Y_{\Delta}(s-\tau)|^{p}+|\bar{Y}_{\Delta}(s)|^{p}+|\bar{Y}_{\Delta}(s-\tau)|^{p})\mbox{d}s\\
&+C\Delta^{1-2\alpha}\Delta^{(1-\alpha)p}\mathbb{E}\int_0^t(|Y_{\Delta}(s)|^{p}+|Y_{\Delta}(s-\tau)|^{p})\mbox{d}s\\
\le& C+C\int_0^t\sup\limits_{0\le u\le s}\mathbb{E}|Y_{\Delta}(u)|^{p}\mbox{d}s.
\end{split}
\end{equation*}
By sorting these equations, we conclude that
\begin{equation*}\label{e3t}
E_3(t)\le C+C\int_0^t\sup\limits_{0\le u\le s}\mathbb{E}|Y_{\Delta}(u)|^{p}\mbox{d}s.
\end{equation*}
Thus, the estimate of $E_1(t)-E_3(t)$ results in
\begin{equation}\label{tildeyt}
\begin{split}
\sup\limits_{0\le u\le t}\mathbb{E}|\tilde{Y}_{\Delta}(u)|^p\le\sup\limits_{0\le u\le t}\mathbb{E}[1+|\tilde{Y}_{\Delta}(u)|^{2}]^{\frac{p}{2}}\le C+C\int_0^t\sup\limits_{0\le u\le s}\mathbb{E}|Y_{\Delta}(u)|^p\mbox{d}s.
\end{split}
\end{equation}
Since $|x-y|^p\ge 2^{1-p}|x|^p-|y|^p$,  we have
\begin{equation*}
\begin{split}
|\tilde{Y}_{\Delta}(t)|^p\ge &2^{1-p}|Y_{\Delta}(t)-D(Y_{\Delta}(t-\tau))|^p-|\theta b_\Delta(Y_{\Delta}(t),Y_{\Delta}(t-\tau))\Delta|^p\\
\ge&2^{1-p}[2^{1-p}|Y_{\Delta}(t)|^p-|D(Y_{\Delta}(t-\tau))|^p]-|\theta b_\Delta(Y_{\Delta}(t),Y_{\Delta}(t-\tau))\Delta|^p.
\end{split}
\end{equation*}
This, combining with (A2) and  (B1), yields that
\begin{equation*}
\begin{split}
|\tilde{Y}_{\Delta}(t)|^p\ge& (2^{2-2p}-\tilde{C}_\Delta)|Y_{\Delta}(t)|^p-(2^{2-p}\kappa^p+\tilde{C}_\Delta)|Y_{\Delta}(t-\tau)|^p-\tilde{C}_\Delta,
\end{split}
\end{equation*}
where $\tilde{C}_\Delta=\theta^p K_5^p3^{p-1}\Delta$. Consequently,
\begin{equation*}
\begin{split}
\sup\limits_{0\le u\le t}\mathbb{E}|Y_{\Delta}(u)|^p\le (2^{2-2p}-2^{2-p}\kappa^p-2\tilde{C}_\Delta)^{-1}\left[\sup\limits_{0\le u\le t}\mathbb{E}|\tilde{Y}_{\Delta}(u)|^p+\tilde{C}_\Delta+(2^{2-p}\kappa^p+\tilde{C}_\Delta)\mathbb{E}\|\xi\|_\infty^p\right].
\end{split}
\end{equation*}
This, together with \eqref{tildeyt},  implies
\begin{equation*}
\begin{split}
\sup\limits_{0\le u\le t}\mathbb{E}|Y_{\Delta}(u)|^p
&\le C+C\int_0^t\sup\limits_{0\le u\le s}\mathbb{E}|Y_{\Delta}(u)|^p\mbox{d}s.
\end{split}
\end{equation*}
Finally, the desired result is obtained by the Gronwall inequality.
\end{proof}

\begin{lem}\label{ytytk}
{\rm
Let (A1)-(A2), (B1)-(B2) hold. Then, we have  any $p\ge2$
\begin{equation*}
\mathbb{E}\left[\sup\limits_{0\le k\le M-1}\sup\limits_{t_k\le t<t_{k+1}}|Y_{\Delta}(t)-Y_{\Delta}(t_k)|^p\right]\le C\Delta^{\frac{p}{2}},
\end{equation*}
where $C$ is a positive constant independent  of  $\Delta$.
}
\end{lem}

\begin{proof}
From the definition of numerical scheme \eqref{ycontinuous}, one sees that for $t\in [t_k, t_{k+1})$,
\begin{equation*}
\tilde{Y}_{\Delta}(t)-\tilde{Y}_{\Delta}(t_k)=\int_{t_k}^t b_\Delta(\bar{Y}_{\Delta}(s), \bar{Y}_{\Delta}(s-\tau))\mbox{d}s+\int_{t_k}^t \sigma_\Delta(\bar{Y}_{\Delta}(s), \bar{Y}_{\Delta}(s-\tau))\mbox{d}W(s).
\end{equation*}
By the elementary inequality $|a+b|^p\le 2^{p-1}(|a|^p+|b|^p), p\ge 1$, we compute
\begin{equation*}
\begin{split}
\mathbb{E}\left[\sup\limits_{t_k\le t<t_{k+1}}|\tilde{Y}_{\Delta}(t)-\tilde{Y}_{\Delta}(t_k)|^p\right]&\le 2^{p-1}\mathbb{E}\left[\sup\limits_{t_k\le t<t_{k+1}}\left|\int_{t_k}^t b_\Delta(\bar{Y}_{\Delta}(s), \bar{Y}_{\Delta}(s-\tau))\mbox{d}s\right|^p\right]\\
&+2^{p-1}\mathbb{E}\left[\sup\limits_{t_k\le t<t_{k+1}}\left|\int_{t_k}^t \sigma_\Delta(\bar{Y}_{\Delta}(s), \bar{Y}_{\Delta}(s-\tau))\mbox{d}W(s)\right|^p\right].
\end{split}
\end{equation*}
With (B1), Lemma \ref{pmoment}, the H\"{o}lder inequality and the BDG inequality, we derive
\begin{equation*}
\begin{split}
&\mathbb{E}\left[\sup\limits_{t_k\le t<t_{k+1}}|\tilde{Y}_{\Delta}(t)-\tilde{Y}_{\Delta}(t_k)|^p\right]
\le 2^{p-1}\Delta^{p-1}\mathbb{E}\int_{t_k}^{t_{k+1}} \left|b_\Delta(\bar{Y}_{\Delta}(s), \bar{Y}_{\Delta}(s-\tau))\right|^p\mbox{d}s\\
&+C\mathbb{E}\left[\int_{t_k}^{t_{k+1}} \left\|\sigma_\Delta(\bar{Y}_{\Delta}(s), \bar{Y}_{\Delta}(s-\tau))\right\|^2\mbox{d}s\right]^{\frac{p}{2}}\\
\le&C\Delta^{(1-\alpha)p}+C\mathbb{E}\left[\int_{t_k}^{t_{k+1}}(1+|\bar{Y}_{\Delta}(s)|^2+|\bar{Y}_{\Delta}(s-\tau)|^2)\mbox{d}s\right]^{\frac{p}{2}}\\
\le &C\Delta^{(1-\alpha)p}+C\Delta^{\frac{p}{2}}\le C\Delta^{\frac{p}{2}}.
\end{split}
\end{equation*}
Denoting by $\tilde{D}(t,t_k):=D(Y_{\Delta}(t-\tau))-D(Y_{\Delta}(t_k-\tau))$, and $\tilde{b}_\Delta(t,t_k):=b_\Delta(Y_{\Delta}(t),Y_{\Delta}(t-\tau))-b_\Delta(Y_{\Delta}(t_k),Y_{\Delta}(t_k-\tau))$, with (A2), we arrive at,
\begin{equation}\label{tildeyt-ytk}
\begin{split}
&|\tilde{Y}_{\Delta}(t)-\tilde{Y}_{\Delta}(t_k)|^p\ge 2^{1-p}|Y_{\Delta}(t)-Y_{\Delta}(t_k)-\tilde{D}(t,t_k)|^p-\theta^p\Delta^p|\tilde{b}_\Delta(t,t_k)|^p\\
\ge&2^{2-2p}|Y_{\Delta}(t)-Y_{\Delta}(t_k)|^p-2^{1-p}\kappa^p|Y_{\Delta}(t-\tau)-Y_{\Delta}(t_k-\tau)|^p-\theta^p\Delta^p|\tilde{b}_\Delta(t,t_k)|^p.
\end{split}
\end{equation}
Obviously, for $0\le t<t_1=\Delta$, we have $t-\tau<t_1-\tau<0$, then we see from \eqref{tildeyt-ytk} and Lemma \ref{pmoment} that
\begin{equation*}
\begin{split}
\mathbb{E}\left[\sup\limits_{0\le t<t_1}|Y_{\Delta}(t)-Y_{\Delta}(t_0)|^p\right]\le& C\mathbb{E}\left[\sup\limits_{0\le t<t_1}|\tilde{Y}_{\Delta}(t)-\tilde{Y}_\Delta(t_0)|^p\right]+C\Delta^{(1-\alpha)p}
\le C\Delta^{\frac{p}{2}}.
\end{split}
\end{equation*}
For $t_1\le t<t_2$, \eqref{tildeyt-ytk} and Lemma \ref{pmoment} lead to
\begin{equation*}
\begin{split}
\mathbb{E}\left[\sup\limits_{t_1\le t<t_2}|Y_{\Delta}(t)-Y_{\Delta}(t_1)|^p\right]\le& C\mathbb{E}\left[\sup\limits_{t_1\le t<t_2}|\tilde{Y}_{\Delta}(t)-\tilde{Y}_{\Delta}(t_1)|^p\right]\\
&+\mathbb{E}\left[\sup\limits_{0\le t<(t_{2-m})\vee 0}|Y_{\Delta}(t)-Y_{\Delta}(t_{1-m})|^p\right]+C\Delta^{(1-\alpha)p}\\
\le&C\Delta^{\frac{p}{2}}.
\end{split}
\end{equation*}
Consequently,  the  induction method yields,
\begin{equation*}
\begin{split}
\mathbb{E}\left[\sup\limits_{t_k\le t<t_{k+1}}|Y_{\Delta}(t)-Y_{\Delta}(t_k)|^p\right]\le& C\mathbb{E}\left[\sup\limits_{t_k\le t<t_{k+1}}|\tilde{Y}_{\Delta}(t)-\tilde{Y}_{\Delta}(t_k)|^p\right]+C\Delta^{(1-\alpha)p}\\
\le&C\Delta^{\frac{p}{2}}.
\end{split}
\end{equation*}
The proof is therefore complete.
\end{proof}

\subsection{Strong Convergence Rate}

The following theorem reveals that the  continuous form $Y_{\Delta}(t)$  of the tamed theta scheme \eqref{discrete1} converges to  the exact solution $X(t)$.

\begin{thm}
{\rm Let (A1)-(A4) and (B1)-(B4) hold,  then it holds that for any $ p\ge 2$,
\begin{equation*}
\mathbb{E}\left(\sup\limits_{0\le t\le T}|X(t)-Y_{\Delta}(t)|^p\right)\le C\Delta^{\alpha p},
\end{equation*}
 where $\alpha$ is defined in (B1) and  $C$ is a positive constant independent of $\DD$. That is, the strong convergence rate of the tamed theta scheme \eqref{discrete1} is $\alpha$.
}
\end{thm}
\begin{proof}
Denote $I(t)=Y_{\Delta}(t)-D(Y_{\Delta}(t-\tau))-\theta b_\Delta(Y_{\Delta}(t),Y_{\Delta}(t-\tau))\Delta-X(t)+D(X(t-\tau))$, then
\begin{equation*}
\begin{split}
I(t)=&I(0)+\int_0^t[b_\Delta(\bar{Y}_{\Delta}(s),\bar{Y}_{\Delta}(s-\tau))-b(X(s),X(s-\tau))]\mbox{d}s\\
&+\int_0^t[\sigma_\Delta(\bar{Y}_{\Delta}(s),\bar{Y}_{\Delta}(s-\tau))-\sigma(X(s),X(s-\tau))]\mbox{d}W(s),
\end{split}
\end{equation*}
where $I(0)=-\theta b_\Delta(\xi(0), \xi(-\tau))\Delta$. An  application of the It\^{o} formula yields,
\begin{equation*}
\begin{split}
|I(t)|^p\le&|I(0)|^p+p\int_0^t |I(s)|^{p-2}\langle I(s),b_\Delta(\bar{Y}_{\Delta}(s),\bar{Y}_{\Delta}(s-\tau))-b(X(s),X(s-\tau))\rangle\mbox{d}s\\
&+\frac{1}{2}p(p-1)\int_0^t |I(s)|^{p-2}\|\sigma_\Delta(\bar{Y}_{\Delta}(s),\bar{Y}_{\Delta}(s-\tau))-\sigma(X(s),X(s-\tau))\|^2\mbox{d}s\\
&+p\int_0^t |I(s)|^{p-2}\langle I(s),(\sigma_\Delta(\bar{Y}_{\Delta}(s),\bar{Y}_{\Delta}(s-\tau))-\sigma(X(s),X(s-\tau)))\mbox{d}W(s)\rangle\\
\le&|I(0)|^p+\sum\limits_{i=1}^7H_i(t),
\end{split}
\end{equation*}
where
\begin{equation*}
\begin{split}
&H_1(t):=p\int_0^t |I(s)|^{p-2}\langle I(s),b_\Delta(\bar{Y}_{\Delta}(s),\bar{Y}_{\Delta}(s-\tau))-b(\bar{Y}_{\Delta}(s),\bar{Y}_{\Delta}(s-\tau))\rangle\mbox{d}s,\\
&H_2(t) :=p\int_0^t |I(s)|^{p-2}\langle I(s),b(\bar{Y}_{\Delta}(s),\bar{Y}_{\Delta}(s-\tau))-b(Y_{\Delta}(s),Y_{\Delta}(s-\tau))\rangle\mbox{d}s,\\
&H_3(t) :=p\int_0^t |I(s)|^{p-2}\langle I(s),b(Y_{\Delta}(s),Y_{\Delta}(s-\tau))-b(X(s),X(s-\tau))\rangle\mbox{d}s,\\
&H_4(t) :=\frac{3}{2}p(p-1)\int_0^t |I(s)|^{p-2}\|\sigma_\Delta(\bar{Y}_{\Delta}(s),\bar{Y}_{\Delta}(s-\tau))-\sigma(\bar{Y}_{\Delta}(s),\bar{Y}_{\Delta}(s-\tau))\|^2\mbox{d}s,\\
&H_5(t) :=\frac{3}{2}p(p-1)\int_0^t |I(s)|^{p-2}\|\sigma(\bar{Y}_{\Delta}(s),\bar{Y}_{\Delta}(s-\tau))-\sigma(Y_{\Delta}(s),Y_{\Delta}(s-\tau))\|^2\mbox{d}s,\\
&H_6(t) :=\frac{3}{2}p(p-1)\int_0^t |I(s)|^{p-2}\|\sigma(Y_{\Delta}(s),Y_{\Delta}(s-\tau))-\sigma(X(s),X(s-\tau))\|^2\mbox{d}s,\\
&H_7(t) :=p\int_0^t |I(s)|^{p-2}\langle I(s),(\sigma_\Delta(\bar{Y}_{\Delta}(s),\bar{Y}_{\Delta}(s-\tau))-\sigma(X(s),X(s-\tau)))\mbox{d}W(s)\rangle.
\end{split}
\end{equation*}
By (A2), (B1), (B4), Lemma \ref{pmoment}, and the H\"{o}lder inequality,
\begin{equation*}
\begin{split}
&\mathbb{E}\left(\sup\limits_{0\le u\le t}H_1(u)\right)\\
&\le C\mathbb{E}\int_0^t|I(s)|^{p}\mbox{d}s+C\mathbb{E}\int_0^t|b_\Delta(\bar{Y}_{\Delta}(s),\bar{Y}_{\Delta}(s-\tau))-b(\bar{Y}_{\Delta}(s),\bar{Y}_{\Delta}(s-\tau))|^p\mbox{d}s\\
&\le C\mathbb{E}\int_0^t[|Y_{\Delta}(s)-X(s)|^p+|Y_{\Delta}(s-\tau)-X(s-\tau)|^p+\theta^p\Delta^p|b_\Delta(Y_{\Delta}(s),Y_{\Delta}(s-\tau))|^p]\mbox{d}s\\
&+C\Delta^{\alpha p}\mathbb{E}\int_0^t(1+|\bar{Y}_{\Delta}(s)|^{2(l+1)p}+|\bar{Y}_{\Delta}(s-\tau)|^{2(l+1)p})\mbox{d}s\\
&\le C\int_0^t\mathbb{E}\left(\sup\limits_{0\le u\le s}|Y_{\Delta}(u)-X(u)|^p\right)\mbox{d}s+C\Delta^{(1-\alpha)p}+C\Delta^{\alpha p}.
\end{split}
\end{equation*}
By (A2), (A4), (B1), Lemmas \ref{pmoment}-\ref{ytytk}, and the H\"{o}lder inequality,
\begin{equation*}
\begin{split}
&\mathbb{E}\left(\sup\limits_{0\le u\le t}H_2(u)\right)\\
&\le C\mathbb{E}\int_0^t|I(s)|^{p}\mbox{d}s+C\mathbb{E}\int_0^t|b(\bar{Y}_{\Delta}(s),\bar{Y}_{\Delta}(s-\tau))-b(Y_{\Delta}(s),Y_{\Delta}(s-\tau))|^p\mbox{d}s\\
\le &C\mathbb{E}\int_0^t[|Y_{\Delta}(s)-X(s)|^p+|Y_{\Delta}(s-\tau)-X(s-\tau)|^p+\theta^p\Delta^p|b_\Delta(Y_{\Delta}(s),Y_{\Delta}(s-\tau))|^p]\mbox{d}s\\
&+C\mathbb{E}\int_0^t(1+|\bar{Y}_{\Delta}(s)|^{l}+|\bar{Y}_{\Delta}(s-\tau)|^{l}+|Y_{\Delta}(s)|^{l}+|Y_{\Delta}(s-\tau)|^{l})^p\\
&\times(|\bar{Y}_{\Delta}(s)-Y_{\Delta}(s)|+|\bar{Y}_{\Delta}(s-\tau)-Y_{\Delta}(s-\tau)|)^p\mbox{d}s\\
\le&C\mathbb{E}\int_0^t (|Y_{\Delta}(s)-X(s)|^p+|Y_{\Delta}(s-\tau)-X(s-\tau)|^p)\mbox{d}s+C\Delta^{(1-\alpha)p}\\
&+C\int_0^t[\mathbb{E}(1+|\bar{Y}_{\Delta}(s)|^{l}+|\bar{Y}_{\Delta}(s-\tau)|^{l}+|Y_{\Delta}(s)|^{l}+|Y_{\Delta}(s-\tau)|^{l})^{2p}]^{\frac{1}{2}}\\
&\times[\mathbb{E}(|\bar{Y}_{\Delta}(s)-Y_{\Delta}(s)|+|\bar{Y}_{\Delta}(s-\tau)-Y_{\Delta}(s-\tau)|)^{2p}]^{\frac{1}{2}}\mbox{d}s\\
\le&C\int_0^t\mathbb{E}\left(\sup\limits_{0\le u\le s}|Y_{\Delta}(u)-X(u)|^p\right)\mbox{d}s+C\Delta^{(1-\alpha)p}+C\Delta^{\frac{p}{2}}.
\end{split}
\end{equation*}
Due to (A2), (A4), (B1), Lemma \ref{pmoment}, and the H\"{o}lder inequality,
\begin{equation*}
\begin{split}
&\mathbb{E}\left(\sup\limits_{0\le u\le t}H_3(u)\right)+\mathbb{E}\left(\sup\limits_{0\le u\le t}H_6(u)\right)\\
\le&C\mathbb{E}\int_0^t |I(s)|^{p-2}[|Y_{\Delta}(s)-X(s)|^{2}+|Y_{\Delta}(s-\tau)-X(s-\tau)|^{2}]\mbox{d}s\\
&+C\mathbb{E}\int_0^t |I(s)|^{p-2}|\theta b_\Delta(Y_{\Delta}(s),Y_{\Delta}(s-\tau))\Delta||b(Y_{\Delta}(s),Y_{\Delta}(s-\tau))-b(X(s),X(s-\tau))|\mbox{d}s\\
\le&C\mathbb{E}\int_0^t[|Y_{\Delta}(s)-X(s)|^{p}+\theta^p\Delta^p|b_\Delta(Y_{\Delta}(s),Y_{\Delta}(s-\tau))|^p]\mbox{d}s\\
&+C\Delta^{1-\alpha}\mathbb{E}\int_0^t|I(s)|^{p-2}(1+|Y_{\Delta}(s)|+|Y_{\Delta}(s-\tau)|)\times\\
&\quad \quad\quad\quad\quad\quad\quad(1+|Y_{\Delta}(s)|^{l}+|Y_{\Delta}(s-\tau)|^{l}+|X(s)|^{l}+|X(s-\tau)|^{l})\times\\
&\quad \quad\quad\quad\quad\quad\quad(|Y_{\Delta}(s)-X(s)|+|Y_{\Delta}(s-\tau)-X(s-\tau)|)\mbox{d}s\\
\le&C\int_0^t\mathbb{E}\left(\sup\limits_{0\le u\le s}|Y_{\Delta}(u)-X(u)|^p\right)\mbox{d}s+C\Delta^{(1-\alpha)p}.
\end{split}
\end{equation*}
In the same way as the estimate of $H_1(t)$ and $H_2(t)$, we  arrive at
\begin{equation*}
\begin{split}
\mathbb{E}\left(\sup\limits_{0\le u\le t}H_4(u)\right)\le&C\int_0^t\mathbb{E}\left(\sup\limits_{0\le u\le s}|Y_{\Delta}(u)-X(u)|^p\right)\mbox{d}s+C\Delta^{(1-\alpha)p}+C\Delta^{\alpha p},
\end{split}
\end{equation*}
\begin{equation*}
\begin{split}
\mathbb{E}\left(\sup\limits_{0\le u\le t}H_5(u)\right)\le C\int_0^t\mathbb{E}\left(\sup\limits_{0\le u\le s}|Y_{\Delta}(u)-X(u)|^p\right)\mbox{d}s+C\Delta^{(1-\alpha)p}+C\Delta^{\frac{p}{2}}.
\end{split}
\end{equation*}
Furthermore, by (B4), Remark \ref{remark2}, Lemmas \ref{pmoment}-\ref{ytytk}, the BDG inequality and the H\"{o}lder inequality, we compute
\begin{equation*}
\begin{split}
& \mathbb{E}\left(\sup\limits_{0\le u\le t}|H_7(u)|\right)\le C\mathbb{E}\left(\int_0^t |I(s)|^{2p-2}\|\sigma_\Delta(\bar{Y}_{\Delta}(s),\bar{Y}_{\Delta}(s-\tau))-\sigma(X(s),X(s-\tau))\|^{2}\mbox{d}s\right)^{\frac{1}{2}}\\
\le&\frac{1}{4}\mathbb{E}\left(\sup\limits_{0\le u\le t }|I(u)|^p\right)+C\int_0^t \mathbb{E}\left(\sup\limits_{0\le u\le s}|Y_{\Delta}(u)-X(u)|^p\right)\mbox{d}s+C\Delta^{(1-\alpha)p}+C\Delta^{\alpha p}+C\Delta^{\frac{p}{2}}.
\end{split}
\end{equation*}
By sorting $H_1(t)-H_7(t)$ together, we derive
\begin{equation*}
\begin{split}
\mathbb{E}\left(\sup\limits_{0\le u\le t}|I(u)|^p\right)\le C\int_0^t\mathbb{E}\left(\sup\limits_{0\le u\le s}|Y_{\Delta}(u)-X(u)|^p\right)\mbox{d}s+C\Delta^{\alpha p}.
\end{split}
\end{equation*}
By the definition of $I(t)$, we have
\begin{equation*}
\begin{split}
|I(t)|^p\ge &2^{1-p}|Y_{\Delta}(t)-X(t)-D(Y_{\Delta}(t-\tau))+D(X(t-\tau))|^p-|\theta b_\Delta(Y_{\Delta}(t),Y_{\Delta}(t-\tau))\Delta|^p\\
\ge&2^{1-p}[2^{1-p}|Y_{\Delta}(t)-X(t)|^p-|D(Y_{\Delta}(t-\tau))-D(X(t-\tau))|^p]-|\theta b_\Delta(Y_{\Delta}(t),Y_{\Delta}(t-\tau))\Delta|^p,
\end{split}
\end{equation*}
this, together with (A2), leads to
\begin{equation*}
\begin{split}
|I(t)|^p\ge& 2^{2-2p}|Y_{\Delta}(t)-X(t)|^p-2^{1-p}\kappa^p|Y_{\Delta}(t-\tau)-X(t-\tau)|^p-|\theta b_\Delta(Y_{\Delta}(t),Y_{\Delta}(t-\tau))\Delta|^p.
\end{split}
\end{equation*}
Taking (B1) and Lemma \ref{pmoment} into consideration yields
\begin{equation*}
\begin{split}
\mathbb{E}\left(\sup\limits_{0\le u\le t}|Y_{\Delta}(u)-X(u)|^p\right)
\le&C\mathbb{E}\left(\sup\limits_{0\le u\le t}|I(u)|^p\right)+C\Delta^{(1-\alpha)p}\\
\le&C\Delta^{\alpha p}+C\int_0^t\mathbb{E}\left(\sup\limits_{0\le u\le s}|Y_{\Delta}(u)-X(u)|^p\right)\mbox{d}s.
\end{split}
\end{equation*}
The desired result follows from the Gronwall inequality.
\end{proof}

\begin{rem}\label{weaker}
{\rm If we replace (A3) and (B2) by the following weaker forms:
\begin{enumerate}
\item[{\bf (A3')}]There exists a positive constant $K_2$ such that for some $p\ge2$
\begin{equation*}
2\< x-D(y), b(x, y)\>+(p-1)\|\sigma(x, y)\|^2\le K_2(1+|x|^2+|y|^2),
\end{equation*}
\item[{\bf (B2')}]There exists a positive constant $\tilde{K}_2$ such that for some $p\ge2$
\begin{equation*}
2\< x-D(y), b_\Delta(x,y)\>+(p-1)\|\sigma_\Delta(x,y)\|^2\le \tilde{K}_2(1+|x|^2+|y|^2),
\end{equation*}
\end{enumerate}
we can also show that under assumptions (A1)-(A2), (A3'), (A4), (B1)-(B2), (B3'), (B4), the tamed theta scheme $Y_\Delta(t)$ converges strongly to the exact solution $X(t)$ with order $\alpha$.
}
\end{rem}

\section{Local One-sided Lipschitz Drift}
In this section, instead of the global one-sided Lipschitz condition (A4), we impose the following local one-sided Lipschitz condition:
\begin{enumerate}
\item[{\bf (A5)}]For every $R>0$, there exists a positive constant $L_R$ such that
\begin{equation*}
\begin{split}
&\quad \< x-D(y)-\bar{x}+D(\bar{y}), b(x, y)-b(\bar{x}, \bar{y})\>\vee\|\sigma(x, y)-\sigma(\bar{x}, \bar{y})\|^2\\
&\le L_R(|x-\bar{x}|^2+|y-\bar{y}|^2)
\end{split}
\end{equation*}
for all $|x|\vee|y|\vee|\bar{x}|\vee|\bar{y}|\le R$.
\end{enumerate}

\begin{rem}\label{bbound}
Due to the continuity of $b(x,y)$, for every $R>0$, there exists a positive constant $\bar{L}_R$ such that
\begin{equation*}
\begin{split}
\sup\limits_{|x|\vee|y|\le R}|b(x,y)|\le \bar{L}_R.
\end{split}
\end{equation*}
\end{rem}

\begin{rem}
{\rm There are many examples such that the assumptions can be  verified. For example, if we set
\begin{equation*}
\begin{split}
D(y)=\frac{1}{4}\cos y,\quad b(x,y)=x-x^3+\cos y,\quad \sigma(x,y)=y\sin x+x\sin y,
\end{split}
\end{equation*}
then assumptions (A2)-(A3) and (A5) hold.
}
\end{rem}

Consider the following tamed theta scheme imposed in Section 2:
\begin{equation*}
\begin{split}
y_{t_{k+1}}-D(y_{t_{k+1-m}})&=y_{t_k}-D(y_{t_{k-m}})+\theta b_\Delta(y_{t_{k+1}}, y_{t_{k+1-m}})\Delta\\
&\quad+(1-\theta) b_\Delta(y_{t_{k}}, y_{t_{k-m}})\Delta+\sigma_\Delta(y_{t_{k}}, y_{t_{k-m}})\Delta W_{t_k}.
\end{split}
\end{equation*}
Generally speaking, for a given $y_{t_k}$,  to guarantee a unique solution $y_{t_{k+1}}$ is to assume that there exists a positive constant $L$ such that
\begin{equation*}
\langle x-D(y)-\bar{x}+D(\bar{y}),b_\Delta(x,y)-b_\Delta(\bar{x},\bar{y})\rangle\le L(|x-\bar{x}|^2+|y-\bar{y}|^2)
\end{equation*}
as in Section 2. Moreover, as shown in Mao and Szpruch \cite{ms13}, this condition is somehow hard to relax. While in our assumption (A5), the drift coefficient $b$ is local one-sided Lipschitz, thus in this case, the tamed drift $b_\Delta$ is hardly to be global one-sided Lipschitz. That is, we do not know if the tamed theta scheme \eqref{discrete1} is well defined under assumptions (A2)-(A3) and (A5). In the following, we will provide an improved tamed theta scheme to ensure the well-posedness of implicit equations.

\subsection{The Improved Tamed Theta Scheme}

For any $R>0$, define a  smooth, non-negative  function such that
\begin{equation*}
\zeta_R(x,y)=\left\{
\begin{array}{lll}
&1,~~&{\rm for}~|x|, |y|\le R,\\
&0,~~&{\rm for}~|x| ~{\rm or}~|y|>R+1,
\end{array}
\right.
\end{equation*}
and $\zeta_{R}(x,y)\le 1$ for all $x,y\in\mathbb{R}^n$. It is obvious that $\zeta_{R}(x,y)$ is Lipschitz with some constant $C_{\zeta}$. Now we introduce the improved tamed theta scheme for \eqref{brownian}. For $k=-m, \cdots, 0$, set $y_{t_k}=\xi(k\Delta)$; For $k=0, 1, \cdots, M-1$, we form
\begin{equation}\label{jingxi}
\begin{split}
y_{t_{k+1}}-D(y_{t_{k+1-m}})&=y_{t_k}-D(y_{t_{k-m}})+\theta \bar{b}_\Delta(y_{t_{k+1}}, y_{t_{k+1-m}})\zeta_R(y_{t_{k+1}}, y_{t_{k+1-m}})\Delta\\
&\quad+(1-\theta) \bar{b}_\Delta(y_{t_{k}}, y_{t_{k-m}})\zeta_R(y_{t_{k}}, y_{t_{k-m}})\Delta+\sigma_\Delta(y_{t_{k}}, y_{t_{k-m}})\Delta W_{t_k},
\end{split}
\end{equation}
where $t_k=k\Delta$, and $\Delta W_{t_k}=W(t_{k+1})-W(t_k)$. Here $\bar{b}_\Delta:\mathbb{R}^n\times\mathbb{R}^n\rightarrow\mathbb{R}^n$ is a continuous function, and  $\sigma_\Delta:\mathbb{R}^n\times\mathbb{R}^n\rightarrow\mathbb{R}^n\otimes\mathbb{R}^d$ is a measurable function. Besides, $\theta\in [0,1]$ is an additional parameter that allows us to control the implicitness of the numerical scheme. Denote
\begin{equation*}
\begin{split}
\tilde{b}_\Delta(x, y)=\bar{b}_\Delta(x, y)\zeta_R(x,y),
\end{split}
\end{equation*}
then, \eqref{jingxi} can be rewritten as
\begin{equation}\label{2discrete1}
\begin{split}
y_{t_{k+1}}-D(y_{t_{k+1-m}})&=y_{t_k}-D(y_{t_{k-m}})+\theta\tilde{ b}_\Delta(y_{t_{k+1}}, y_{t_{k+1-m}})\Delta\\
&\quad+(1-\theta) \tilde{b}_\Delta(y_{t_{k}}, y_{t_{k-m}})\Delta+\sigma_\Delta(y_{t_{k}}, y_{t_{k-m}})\Delta W_{t_k},
\end{split}
\end{equation}
which is exactly the form of \eqref{discrete1}. According to  \eqref{2discrete1} we define  $\bar{Y}_{\Delta}(t)$, $Y_{\Delta}(t)$, $\tilde{Y}_{\Delta}(t)$, $Z_\Delta(t)$ by using the same notation as  in Section 2. Instead of constraints on $\tilde{b}_\Delta(x,y)$, we impose some assumptions on $\bar{b}_\Delta(x,y)$ and $ \sigma_\Delta(x,y)$. Assume that there exists   an $\alpha\in(0, 1/2]$ such that for any $x, y, \bar{x}, \bar{y} \in\R^n$, the following conditions hold:
\begin{enumerate}
\item[{\bf (C1)}]There exists a positive constant $K_5\ge 1$ such that
\begin{equation*}
|\bar{b}_\Delta(x,y)|\le \min(K_5\Delta^{-\alpha}(1+|x|+|y|), |b(x,y)|),
\end{equation*}
and
\begin{equation*}
\|\sigma_\Delta(x,y)\|^2\le \min(K_5\Delta^{-\alpha}(1+|x|^2+|y|^2), \|\sigma(x,y)\|^2).
\end{equation*}
\item[{\bf (C2)}]There exists a positive constant $\tilde{K}_2$ such that
\begin{equation*}
\< x-D(y), \bar{b}_\Delta(x,y)\>\vee\|\sigma_\Delta(x,y)\|^2\le \tilde{K}_2(1+|x|^2+|y|^2).
\end{equation*}
\item[{\bf (C3)}] For any $R>0$, there exists a positive constant $M_R$ such that
\begin{equation*}
\begin{split}
&\quad \< x-D(y)-\bar{x}+D(\bar{y}), \bar{b}_\Delta(x, y)-\bar{b}_\Delta(\bar{x}, \bar{y})\>\le M_R(|x-\bar{x}|^2+|y-\bar{y}|^2).
\end{split}
\end{equation*}
for  all $|x|\vee|y|\vee|\bar{x}|\vee|\bar{y}|\le R$.
\item[{\bf (C4)}] For any $R>0$, there exists a positive constant $N_R$ such that
\begin{equation*}
\sup_{|x|\vee|y|\le R}\left[|b(x,y)-\bar{b}_\Delta(x,y)|^{p}\vee\|\sigma(x,y)-\sigma_\Delta(x,y)\|^{p}\right]\le N_R\Delta^{\alpha p}\rightarrow 0\quad\mbox{as}\quad \Delta\rightarrow 0.
\end{equation*}
\end{enumerate}

\begin{lem}\label{diff}
{\rm Let (A2), (C1)-(C4) hold, then $\tilde{b}_\Delta$ satisfies (C1), (C2), (C4) and the following (C3'):
\begin{enumerate}
\item[{\bf (C3')}]There exists an $\bar{M}_{R_0}$ such that for all $x, y, \bar{x}, \bar{y} \in\R^n$
\begin{equation*}
\begin{split}
&\quad \< x-D(y)-\bar{x}+D(\bar{y}), \tilde{b}_\Delta(x, y)-\tilde{b}_\Delta(\bar{x}, \bar{y})\>\le \bar{M}_{R_0}(|x-\bar{x}|^2+|y-\bar{y}|^2),
\end{split}
\end{equation*}
where $\bar{M}_{R_0}=M_{R_0}+2C_\zeta \bar{L}_{R_0}$.
\end{enumerate}
}
\end{lem}
\begin{proof}
By the relationship between $\tilde{b}_\Delta$ and $\bar{b}_{\Delta}$, (C1) and (C2) can be verified easily. Noting that for $|x|\vee|y|\le R$, $\zeta_R(x,y)=1$, thus we get
\begin{equation*}
\begin{split}
&\sup_{|x|\vee|y|\le R}\left[|b(x,y)-\tilde{b}_\Delta(x,y)|^{p}\right]=\sup_{|x|\vee|y|\le R}\left[|b(x,y)-\bar{b}_\Delta(x,y)\zeta_\Delta(x,y)|^{p}\right]\\
=&\sup_{|x|\vee|y|\le R}\left[|b(x,y)-\bar{b}_\Delta(x,y)|^{p}\right]\le N_R\Delta^{\alpha p}\rightarrow 0, \quad\mbox{as}\quad \Delta\rightarrow 0,
\end{split}
\end{equation*}
then (C4) holds for $\tilde{b}_\Delta(x,y)$. Now we are going to check (C3'). Divide it into four cases.\\
{\bf Case a:} None of $|x|,|y|,|\bar{x}|,|\bar{y}|$ bigger than $R+1$. In this case, we see $0\le\zeta_{R}(x,y),\zeta_{R}(\bar{x},\bar{y})\le1$. Rewrite $\tilde{b}_\Delta$ with $\bar{b}_\Delta$, we have
\begin{equation*}
\begin{split}
&\langle x-D(y)-\bar{x}+D(\bar{y}), \tilde{b}_\Delta(x, y)-\tilde{b}_\Delta(\bar{x}, \bar{y})\rangle\\
=&\langle x-D(y)-\bar{x}+D(\bar{y}), \bar{b}_\Delta(x, y)-\bar{b}_\Delta(\bar{x}, \bar{y})\rangle\zeta_R(x,y)\\
&+\langle x-D(y)-\bar{x}+D(\bar{y}), (\zeta_R(x,y)-\zeta_R(\bar{x},\bar{y}))\bar{b}_\Delta(\bar{x}, \bar{y})\rangle\\
=:&q_1+q_2.
\end{split}
\end{equation*}
Since $0\le\zeta_R(x,y)\le 1$, thus by (C3),
\begin{equation*}
\begin{split}
q_1\le M_{R+1}(|x-\bar{x}|^2+|y-\bar{y}|^2).
\end{split}
\end{equation*}
Further, noting that $\zeta_R$ is Lipschitz with constant $C_\zeta$ and for $|\bar{x}|\vee|\bar{y}|\le R+1$, we see from (C1) and Remark \ref{bbound} that $|\bar{b}_\Delta(\bar{x}, \bar{y})|\le|b(\bar{x},\bar{y})|\le \bar{L}_{R+1}$, then (A2) leads to
\begin{equation*}
\begin{split}
q_2\le 2C_\zeta \bar{L}_{R+1}(|x-\bar{x}|^2+|y-\bar{y}|^2).
\end{split}
\end{equation*}
Combining the estimation of $q_1$ and $q_2$, we get
\begin{equation*}
\begin{split}
&\quad \< x-D(y)-\bar{x}+D(\bar{y}), \tilde{b}_\Delta(x, y)-\tilde{b}_\Delta(\bar{x}, \bar{y})\>\le (M_{R+1}+2C_\zeta \bar{L}_{R+1})(|x-\bar{x}|^2+|y-\bar{y}|^2).
\end{split}
\end{equation*}
{\bf Case b:} One of $|x|,|y|,|\bar{x}|,|\bar{y}|$ bigger than $R+1$. Assume $|x|>R+1$ and $|y|,|\bar{x}|,|\bar{y}|\le R+1$. In this case, we have $\zeta_{R}(x,y)=0$ and $0\le\zeta_{R}(\bar{x},\bar{y})\le1$. Similar to Case a, we have
\begin{equation*}
\begin{split}
&\quad \< x-D(y)-\bar{x}+D(\bar{y}), \tilde{b}_\Delta(x, y)-\tilde{b}_\Delta(\bar{x}, \bar{y})\>\le 2C_\zeta \bar{L}_{R+1}(|x-\bar{x}|^2+|y-\bar{y}|^2).
\end{split}
\end{equation*}
{\bf Case c:} Two of $|x|,|y|,|\bar{x}|,|\bar{y}|$ bigger than $R+1$. We divide it into two cases.\\
{\bf{i)}:} Both $|x|,|y|$ bigger than $R+1$ or both $|\bar{x}|,|\bar{y}|$ bigger than $R+1$. Consider one of the case $|x|,|y|>R+1$ while $|\bar{x}|,|\bar{y}|\le R+1$. It is obvious that $\zeta_R(x,y)=0$ and $0\le\zeta_{R}(\bar{x},\bar{y})\le1$. By taking similar steps as Case a, we can get
\begin{equation*}
\begin{split}
&\quad \< x-D(y)-\bar{x}+D(\bar{y}), \tilde{b}_\Delta(x, y)-\tilde{b}_\Delta(\bar{x}, \bar{y})\>\le 2C_\zeta \bar{L}_{R+1}(|x-\bar{x}|^2+|y-\bar{y}|^2).
\end{split}
\end{equation*}
{\bf \bf{ii)}:} One of $|x|,|y|$ bigger than $R+1$ and one of $|\bar{x}|,|\bar{y}|$ bigger than $R+1$. Consider the case of $|x|>R+1, |y|\le R+1,|\bar{x}|>R+1,|\bar{y}|\le R+1$. Then $\zeta_R(x,y)=\zeta_{R}(\bar{x},\bar{y})=0$ and $\langle x-D(y)-\bar{x}+D(\bar{y}), \tilde{b}_\Delta(x, y)-\tilde{b}_\Delta(\bar{x}, \bar{y})\rangle=0$. \\
{\bf Case d:} Three or four of $|x|,|y|,|\bar{x}|,|\bar{y}|$ bigger than $R+1$. Since we have $\zeta_R(x,y)=\zeta_{R}(\bar{x},\bar{y})=0$, the result is obvious.\\
Taking Cases a-d into consideration, there exists an $\bar{M}_{R_0}$ such that (C3') satisfies for all $x, y, \bar{x}, \bar{y} \in\R^n$.

\end{proof}

\begin{rem}
{\rm Lemma \ref{diff} shows that with assumptions (C1)-(C3), \eqref{2discrete1} is well defined under some constraints on time step $\Delta$. It is worth mentioning that (C3) and (C3') are merely used to guarantee the uniqueness of numerical solutions.
}
\end{rem}

\begin{rem}
{\rm Under assumptions (A2)-(A3), (A5), we now give an example such that the set of sequences of functions satisfy (C1)-(C4). Let $b(x, y), \sigma(x,y)$ be one-dimensional and define
\begin{equation*}
\bar{b}_\Delta(x,y)=\frac{1}{1+\Delta^{\alpha}|b(x,y)|+\Delta^{\alpha/2}\|\sigma(x,y)\|}b(x,y),
\end{equation*}
and
\begin{equation*}
\sigma_\Delta(x,y)=\frac{1}{1+\Delta^{\alpha}|b(x,y)|+\Delta^{\alpha/2}\|\sigma(x,y)\|}\sigma(x,y),
\end{equation*}
for any  $x, y\in\mathbb{R}$. It is easy to see $|\bar{b}_\Delta(x, y)|\le |b(x,y)|$, and on the other hand, we have
\begin{equation*}
|\bar{b}_\Delta(x,y)|=\frac{\Delta^{-\alpha}|b(x,y)|}{\Delta^{-\alpha}+|b(x,y)|+\Delta^{-\alpha/2}\|\sigma(x,y)\|}\le K_5\Delta^{-\alpha}\le K_5\Delta^{-\alpha}(1+|x|+|y|),
\end{equation*}
and
\begin{equation*}
\|\sigma_\Delta(x,y)\|^2=\left[\frac{\Delta^{-\alpha/2}\sigma(x,y)}{\Delta^{-\alpha/2}+\Delta^{\alpha/2}|b(x,y)|+\|\sigma(x,y)\|}\right]^2\le K_5\Delta^{-\alpha}\le K_5\Delta^{-\alpha}(1+|x|^2+|y|^2).
\end{equation*}
Furthermore, due to (A3),
\begin{equation*}\label{b2}
\begin{split}
&\quad \< x-D(y), \bar{b}_\Delta(x,y)\>=\frac{\< x-D(y), b(x,y)\>}{1+\Delta^{\alpha}|b(x,y)|+\Delta^{\alpha/2}\|\sigma(x,y)\|}\le K_2(1+|x|^2+|y|^2).
\end{split}
\end{equation*}
That is to say, (C2) is satisfied.  In order to show (C3), we have to divide it into several cases. Denote by $\Gamma(x,y)=1+\Delta^{\alpha}|b(x,y)|+\Delta^{\alpha/2}\|\sigma(x,y)\|$. \\
{\bf Case a:} $b(x,y)\cdot b(\bar{x},\bar{y})<0$. We divide this into four classes. \\
{\bf{i)}:} For $b(x,y)>0, b(\bar{x},\bar{y})<0$ and $x-D(y)-\bar{x}+D(\bar{y})\ge0$,
\begin{equation*}
\begin{split}
&\< x-D(y)-\bar{x}+D(\bar{y}), \bar{b}_\Delta(x, y)-\bar{b}_\Delta(\bar{x}, \bar{y})\>=\left\< x-D(y)-\bar{x}+D(\bar{y}), \frac{b(x, y)}{\Gamma(x,y)}-\frac{b(\bar{x}, \bar{y})}{\Gamma(\bar{x},\bar{y})}\right\>\\
\le& \< x-D(y)-\bar{x}+D(\bar{y}), b(x, y)-b(\bar{x}, \bar{y})\>\le L_R(|x-\bar{x}|^2+|y-\bar{y}|^2).
\end{split}
\end{equation*}
{\bf{ii)}:} For $b(x,y)>0, b(\bar{x},\bar{y})<0$ and $x-D(y)-\bar{x}+D(\bar{y})<0$, the result is obvious. \\
{\bf{iii)}:} For $b(x,y)<0, b(\bar{x},\bar{y})>0$ and $x-D(y)-\bar{x}+D(\bar{y})<0$,
\begin{equation*}
\begin{split}
&\< x-D(y)-\bar{x}+D(\bar{y}), \bar{b}_\Delta(x, y)-\bar{b}_\Delta(\bar{x}, \bar{y})\>\\
\le& \< x-D(y)-\bar{x}+D(\bar{y}), b(x, y)-b(\bar{x}, \bar{y})\>\le L_R(|x-\bar{x}|^2+|y-\bar{y}|^2).
\end{split}
\end{equation*}
{\bf \bf{iv)}:} For $b(x,y)<0, b(\bar{x},\bar{y})>0$ and $x-D(y)-\bar{x}+D(\bar{y})\ge0$, the result is also obvious. \\
{\bf Case b:} $b(x,y)\cdot b(\bar{x},\bar{y})>0$. We compute
\begin{equation*}
\begin{split}
&\< x-D(y)-\bar{x}+D(\bar{y}), \bar{b}_\Delta(x, y)-\bar{b}_\Delta(\bar{x}, \bar{y})\>=\left\< x-D(y)-\bar{x}+D(\bar{y}), \frac{b(x, y)-b(\bar{x}, \bar{y})}{\Gamma(x,y)\Gamma(\bar{x},\bar{y})}\right\>\\
&+\left\< x-D(y)-\bar{x}+D(\bar{y}),\frac{\Delta^{\alpha}[b(x, y)|b(\bar{x}, \bar{y})|-b(\bar{x}, \bar{y})|b(x, y)|]}{\Gamma(x,y)\Gamma(\bar{x},\bar{y})}\right\>\\
&+\left\< x-D(y)-\bar{x}+D(\bar{y}),\frac{\Delta^{\alpha/2}\|\sigma(\bar{x},\bar{y})\|[b(x, y)-b(\bar{x}, \bar{y})]}{\Gamma(x,y)\Gamma(\bar{x},\bar{y})}\right\>\\
&+\left\< x-D(y)-\bar{x}+D(\bar{y}),\frac{\Delta^{\alpha/2}b(\bar{x}, \bar{y})[\|\sigma(\bar{x},\bar{y})\|-\|\sigma(x, y)\|]}{\Gamma(x,y)\Gamma(\bar{x},\bar{y})}\right\>\\
:=&\bar{q}_1+\bar{q}_2+\bar{q}_3+\bar{q}_4.
\end{split}
\end{equation*}
Obviously, $\bar{q}_2=0$. Noticing that $\Gamma(x,y)\ge1, \Gamma(\bar{x},\bar{y})\ge1$ and $0<\frac{\Delta^{\alpha/2}\|\sigma(\bar{x},\bar{y})\|}{\Gamma(\bar{x},\bar{y})}\le 1$, we then derive from (A2), (A5) and Remark \ref{bbound} that
\begin{equation*}
\begin{split}
\bar{q}_1+\bar{q}_3\le 2L_R(|x-\bar{x}|^2+|y-\bar{y}|^2),
\end{split}
\end{equation*}
and
\begin{equation*}
\begin{split}
\bar{q}_4\le& \frac{1}{2}|x-D(y)-\bar{x}+D(\bar{y})|^2+\frac{\Delta^{\alpha}|b(\bar{x}, \bar{y})|^2[\|\sigma(\bar{x},\bar{y})\|-\|\sigma(x, y)\|]^2}{2\Gamma^2(x,y)\Gamma^2(\bar{x},\bar{y})}\\
\le&|x-\bar{x}|^2+|y-\bar{y}|^2+|b(x,y)|\|\sigma(\bar{x},\bar{y})-\sigma(x, y)\|^2\\
\le&(1+L_R\bar{L}_R)(|x-\bar{x}|^2+|y-\bar{y}|^2).
\end{split}
\end{equation*}
This shows that (C3) is satisfied. Thanks to (A3) and Remark \ref{bbound}, we see that
\begin{equation*}
\begin{split}
\sup_{|x|\vee|y|\le R}|b(x,y)-\bar{b}_\Delta(x,y)|^{p}
\le \Delta^{\alpha p}\sup_{|x|\vee|y|\le R}\frac{(|b(x,y)|+\|\sigma(x,y)\|^{2})^p|b(x,y)|^{p}}{(1+\Delta^{\alpha}|b(x,y)|+\Delta^{\alpha}\|\sigma(x,y)\|^2)^{p}}\le C\Delta^{\alpha p}\rightarrow 0,
\end{split}
\end{equation*}
In the same way we can show that  the diffusion coefficient satisfies (C4).
}
\end{rem}

Condition (C3') shows that $\tilde{b}_\Delta$ is global one-sided Lipschitz. According to the monotone operator, the implicit scheme \eqref{2discrete1} is well defined with $\theta\bar{M}_{R_0}\Delta<1$. Define $\Delta_3=\frac{1}{\theta\bar{M}_{R_0}}$ for $\theta\in(0,1]$. Thus in the following section, we set $\Delta^\star\in(0,\Delta_3\wedge\Delta_2)$, and let $0<\Delta\le\Delta^\star$ for $\theta\in(0,1]$ while for $\theta=0$, we set $\Delta\in(0,1)$.

\subsection{Convergence of the Numerical Solutions}

We need the following lemma.
\begin{lem}\label{2exactexist}
{\rm Let (A1)-(A3) and (A5) hold, then it holds that
\begin{equation*}
\sup\limits_{0\le t\le T}\mathbb{E}|X(t)|^p\vee \sup\limits_{0\le t\le T}\mathbb{E}|Y_\Delta(t)|^p\le C,
\end{equation*}
for any $p\ge 2$.
}
\end{lem}

\begin{rem}
{\rm Since Lemmas \ref{pmoment} and \ref{ytytk} depend only on assumptions (A1)-(A2) and (B1)-(B2), in this section, $\bar{b}_\Delta$ and $\sigma_\Delta$ satisfy (C1)-(C2), which implies that the corresponding $\tilde{b}_\Delta$ and $\sigma_\Delta$ also satisfy (B1)-(B2). Thus, Lemmas \ref{pmoment} and \ref{ytytk} proposed in Section 2 still hold in this section.
}
\end{rem}

We now state the main result in this Section.

\begin{thm}
{\rm
Let (A1)-(A3), (A5) and (C1)-(C4) hold, then the continuous form $Y_{\Delta}(t)$ of the tamed theta scheme \eqref{2discrete1} converges strongly to the exact solution $X(t)$  of \eqref{brownian},  that is,
\begin{equation*}
\lim\limits_{\Delta\rightarrow 0}\sup\limits_{0\le t\le T}\mathbb{E}|X(t)-Y_{\Delta}(t)|^2=0.
\end{equation*}

}
\end{thm}

\begin{proof}
Denote $e(t)=Y_{\Delta}(t)-X(t)$, and for any $R>0$ define the following stopping time
\begin{equation*}
\tau_R=\inf\{t\ge 0: |Y_{\Delta}(t)|\ge R\}, \rho_R=\inf\{t\ge 0: |X(t)|\ge R\}, \nu_R=\tau_R\wedge\rho_R.
\end{equation*}
For any $\eta>0$, by the Young inequality,
\begin{equation}\label{ideltau2}
\begin{split}
\sup\limits_{0\le u\le T}\mathbb{E}|e(u)|^2=&\sup\limits_{0\le u\le T}\mathbb{E}(|e(u)|^2{\bf I}_{\{\tau_R>T,\rho_R>T\}})+\sup\limits_{0\le u\le T}\mathbb{E}(|e(u)|^2{\bf I}_{\{\tau_R\le T~{\rm or}~\rho_R\le T\}})\\
\le&\sup\limits_{0\le u\le T}\mathbb{E}(|e(u\wedge\nu_R)|^2{\bf I}_{\{\nu_R>T\}})+\frac{2\eta}{p}\sup\limits_{0\le u\le T}\mathbb{E}|e(u)|^p\\
&+\frac{p-2}{p\eta^{\frac{2}{p-2}}}\mathbb{P}(\tau_R\le T~{\rm or}~\rho_R\le T).
\end{split}
\end{equation}
Due to Lemma \ref{pmoment},
\begin{equation*}
\begin{split}
\mathbb{P}(\tau_R\le T)=\mathbb{E}\left({\bf I}_{\{\tau_R\le T\}}\frac{|Y_\Delta(\tau_R)|^p}{R^p}\right)\le\frac{1}{R^p}\sup\limits_{0\le u \le T}\mathbb{E}|Y_{\Delta}(u)|^p\le\frac{C}{R^p},
\end{split}
\end{equation*}
where here and in the following, we emphasize that $C$ is a positive constant independent of $\Delta, R$ and $\varepsilon$, while $C_R$ will be  a positive constant depending on $R$. Similarly, we derive from Lemma \ref{2exactexist} that
\begin{equation}\label{ptaur}
\begin{split}
\mathbb{P}(\tau_R\le T~{\rm or}~\rho_R\le T)\le\mathbb{P}(\tau_R\le T)+\mathbb{P}(\rho_R\le T)\le\frac{2C}{R^p}.
\end{split}
\end{equation}
On the other hand, Lemma \ref{pmoment} and Lemma \ref{2exactexist} yield
\begin{equation}\label{eideltaup}
\begin{split}
\sup\limits_{0\le u\le T}\mathbb{E}|e(u)|^p\le 2^{p-1}\sup\limits_{0\le u \le T}\mathbb{E}(|Y_{\Delta}(u)|^p+|X(u)|^p)\le C.
\end{split}
\end{equation}
Denote by $I(t)=Y_{\Delta}(t)-D(Y_{\Delta}(t-\tau))-\theta \tilde{b}_\Delta(Y_{\Delta}(t),Y_{\Delta}(t-\tau))\Delta-X(t)+D(X(t-\tau))$.
Applying the It\^{o} formula,
\begin{equation*}
\begin{split}
&\mathbb{E}|I(T\wedge\nu_R)|^2=|I(0)|^2+2\mathbb{E}\int_0^{T\wedge\nu_R}\langle I(s),\tilde{b}_\Delta(\bar{Y}_{\Delta}(s), \bar{Y}_{\Delta}(s-\tau))-b(X(s),X(s-\tau))\rangle\mbox{d}s\\
&+\mathbb{E}\int_0^{T\wedge\nu_R}\|\sigma_\Delta(\bar{Y}_{\Delta}(s), \bar{Y}_{\Delta}(s-\tau))-\sigma(X(s),X(s-\tau))\|^2\mbox{d}s\\
\le&|I(0)|^2+2\mathbb{E}\int_0^{T\wedge\nu_R}\langle \bar{Y}_{\Delta}(s)-D(\bar{Y}_{\Delta}(s-\tau))-X(s)+D(X(s-\tau)),\\
&~~~~~~~~~~~~~~~b(\bar{Y}_{\Delta}(s),\bar{Y}_{\Delta}(s-\tau))-b(X(s),X(s-\tau))\rangle\mbox{d}s\\
&+2\mathbb{E}\int_0^{T\wedge\nu_R}\langle \bar{Y}_{\Delta}(s)-D(\bar{Y}_{\Delta}(s-\tau))-X(s)+D(X(s-\tau)),\\
&~~~~~~~~~~~~~~~\tilde{b}_\Delta(\bar{Y}_{\Delta}(s),\bar{Y}_{\Delta}(s-\tau))-b(\bar{Y}_{\Delta}(s),\bar{Y}_{\Delta}(s-\tau))\rangle\mbox{d}s\\
&+2\mathbb{E}\int_0^{T\wedge\nu_R}\langle Y_{\Delta}(s)-D(Y_{\Delta}(s-\tau))-\bar{Y}_{\Delta}(s)+D(\bar{Y}_{\Delta}(s-\tau)),\\
&~~~~~~~~~~~~~~~\tilde{b}_\Delta(\bar{Y}_{\Delta}(s),\bar{Y}_{\Delta}(s-\tau))-b(X(s),X(s-\tau))\rangle\mbox{d}s\\
&-2\theta\Delta\mathbb{E}\int_0^{T\wedge\nu_R}\langle \tilde{b}_\Delta(Y_{\Delta}(s),Y_{\Delta}(s-\tau)),\tilde{b}_\Delta(\bar{Y}_{\Delta}(s), \bar{Y}_{\Delta}(s-\tau))-b(X(s),X(s-\tau))\rangle\mbox{d}s\\
&+2\mathbb{E}\int_0^{T\wedge\nu_R}\|\sigma(\bar{Y}_{\Delta}(s),\bar{Y}_{\Delta}(s-\tau))-\sigma(X(s),X(s-\tau))\|^2\mbox{d}s\\
&+2\mathbb{E}\int_0^{T\wedge\nu_R}\|\sigma_\Delta(\bar{Y}_{\Delta}(s),\bar{Y}_{\Delta}(s-\tau))-\sigma(\bar{Y}_{\Delta}(s),\bar{Y}_{\Delta}(s-\tau))\|^2\mbox{d}s\\
\le&|I(0)|^2+\sum\limits_{i=1}^6I_i(T),
\end{split}
\end{equation*}
where $I(0)=-\theta b_\Delta(\xi(0),\xi(-\tau))\Delta$. By assumption (A5) and Lemma \ref{ytytk},
\begin{equation}\label{edelta15}
\begin{split}
\sup\limits_{0\le u\le T}(|I_1(u)+I_5(u)|)\le& C_R\int_0^{T\wedge\nu_R}[|\bar{Y}_{\Delta}(s)-X(s)|^2+|\bar{Y}_{\Delta}(s-\tau)-X(s-\tau)|^2]\mbox{d}s\\
\le& C_R\Delta+\int_0^{t}\sup\limits_{0\le u\le s}\mathbb{E}|e(u\wedge\nu_R)|^2\mbox{d}s.
\end{split}
\end{equation}
 Obviously, due to (C4), Lemma \ref{exactexist}  and Lemma \ref{pmoment},
\begin{equation}\label{edelta26}
\begin{split}
\sup\limits_{0\le u\le T}(|I_2(u)|+|I_6(u)|)\le C_R\Delta^{\alpha}+C_R\Delta^{2\alpha}.
\end{split}
\end{equation}
By (A2) and Lemma \ref{ytytk}, we obtain,
\begin{equation}\label{edelta3}
\begin{split}
\sup\limits_{0\le u\le T}|I_3(u)|\le &C\int_0^{T\wedge\nu_R}\left(\mathbb{E}[|Y_{\Delta}(s)-\bar{Y}_{\Delta}(s)|^2+|Y_{\Delta}(s-\tau)-\bar{Y}_{\Delta}(s-\tau)|^2]\right)^{\frac{1}{2}}\\
&\left(\mathbb{E}|\tilde{b}_\Delta(\bar{Y}_{\Delta}(s),\bar{Y}_{\Delta}(s-\tau))-b(X(s),X(s-\tau))|^2\right)^{\frac{1}{2}}\mbox{d}s\le C_R\Delta^{\frac{1}{2}}.
\end{split}
\end{equation}
Furthermore, by Remark \ref{bbound} and (C1)
\begin{equation}\label{edelta4}
\begin{split}
\sup\limits_{0\le u\le T}|I_4(u)|\le C_R\Delta.
\end{split}
\end{equation}
Due to \eqref{edelta15}-\eqref{edelta4}, we see
\begin{equation}\label{tl1}
\begin{split}
&\sup\limits_{0\le u\le T}\mathbb{E}|I(u\wedge\nu_R)|^2\le C_R\int_0^{T}\sup\limits_{0\le u\le s}\mathbb{E}|e(u\wedge\nu_R)|^2\mbox{d}s+C_R\Delta^\alpha.
\end{split}
\end{equation}
With assumption (A2), one has
\begin{equation*}
\begin{split}
&\sup\limits_{0\le u\le T}\mathbb{E}|e(u\wedge\nu_R)|^2\le C\sup\limits_{0\le u\le T}\mathbb{E}|I(u\wedge\nu_R)|^2+\kappa\sup\limits_{0\le u\le T}\mathbb{E}|e(u\wedge\nu_R-\tau)|^2+C_R\Delta^2.
\end{split}
\end{equation*}
This implies
\begin{equation}\label{tl2}
\begin{split}
&\sup\limits_{0\le u\le T}\mathbb{E}|e(u\wedge\nu_R)|^2\le C\sup\limits_{0\le u\le T} \mathbb{E}|I(u\wedge\nu_R)|^2+C_R\Delta^2.
\end{split}
\end{equation}
Applying the Gronwall inequality, we derive from \eqref{tl1} and \eqref{tl2} that
\begin{equation}\label{tl3}
\begin{split}
&\sup\limits_{0\le u\le T}\mathbb{E}|e(u\wedge\nu_R)|^2\le C_R\Delta^\alpha.
\end{split}
\end{equation}
Thus, combining \eqref{ptaur}, \eqref{eideltaup} and \eqref{tl3}, we see from \eqref{ideltau2} that for any given $\epsilon>0$, one can choose $\eta$ small enough such that
\begin{equation*}
\frac{2\eta}{p}C<\frac{\epsilon}{3},
\end{equation*}
and then $R$ big enough such that
\begin{equation*}
\frac{p-2}{p\eta^{\frac{2}{p-2}}}\frac{2C}{R^p}<\frac{\epsilon}{3},
\end{equation*}
finally $\Delta$ small enough to satisfy
\begin{equation*}
C_R\Delta^\alpha<\frac{\epsilon}{3}.
\end{equation*}
Therefore, we arrive at
\begin{equation*}
\begin{split}
\sup\limits_{0\le u\le T}\mathbb{E}|Y_{\Delta}(u)-X(u)|^2\rightarrow 0~{\rm as}~\Delta\rightarrow 0,
\end{split}
\end{equation*}
as required.
\end{proof}


\begin{thebibliography}{20}
\bibitem{Hai16} Hairer, M., Hutzenthaler, M., Jentzen, A., Loss of regularity for Kolmogorov equations, {\it The Annals of Probab.,} {\bf 43} (2015), 468-527.

\bibitem{hms02} Higham, D.J., Mao, X.R., Stuart, A.M., Strong convergence of Euler-type methods for nonlinear stochastic differential equations, {\it SIAM J. Numer. Anal.}, {\bf 40} (2002), 1041-1063.

\bibitem{hjk11} Hutzenthaler, M., Jentzen, A., Kloeden, P.E., Strong and weak divergence in finite time of Euler¡¯s method for stochastic differential equations with non-globally Lipschitz continuous coefficients, {\it Proc. R. Soc. Lond. Ser. A Math. Phys. Eng. Sci.}, {\bf 467} (2011), 1563-1576.

\bibitem{hjk12} Hutzenthaler, M., Jentzen, A., Kloeden, P.E., Strong convergence of an explicit numerical method for SDEs with nonglobally Lipschitz continuous coefficients, {\it The Annals Probab.}, {\bf  22} (2012), 1611-1641.


\bibitem{jiyuan16} Ji, Y.T., Yuan, C.G., Tamed EM scheme of neutral stochastic differential delay equations, arXiv:1603.06747v1, 2016.



\bibitem{kp92} Kloeden, P.E., Platen, E., {\it Numerical Solution of Stochastic Differential Equations}, Applications of Mathematics, Springer, Berlin, 1992.

\bibitem{ly15} Lan, G.Q., Yuan C.G., Exponential stability of the exact solutions and $\theta$-EM approximations to neutral SDDEs with Markov switching, {\it J. Comp. Appl. Math.}, {\bf 285} (2015), 230-242.

\bibitem{ms13} Mao, X.R., Szpruch, L., Strong convergence and stability of implicit numerical methods for stochastic differential equations with non-globally Lipschitz continuous coefficients, {\it J. Comp. Appl. Math.}, {\bf 238} (2013), 14-28.


\bibitem{m55} Maruyama, G., Continuous Markov processes and stochastic equations, {\it Rend. Circ. Mat. Palermo}, {\bf 4} (1955), 48-90.

\bibitem{m95} Milstein, G.N., {\it Numerical Integration of Stochastic Differential Equations}, Kluwer Academic, Dordrecht, 1955.

\bibitem{sabanis13} Sabanis, S., A note on tamed Euler approximations, {\it Electron. Commun. Probab.}, {\bf 18} (2013), 1-10.

\bibitem{sabanis15} Sabanis, S., Euler approximations with varying coefficients: the case of superlinearly growing diffusion coefficients, {\it Annal. Appl. Probab.}, {\bf 26} (2016), 2083-2105.

\bibitem{wm08} Wu, F.K., Mao X.R., Numerical solutions of neutral stochastic functional differential equations, {\it SIAM J. Numer. Anal.}, {\bf 46} (2008), 1821-1841.

\bibitem{zei} Zeidler, E., {\it Nonlinear Functional Analysis and its Applications}, Springer-Verlag, New York, 1990.

\bibitem{z15} Zhou, S.B., Exponential stability of numerical solution to neutral stochastic functional differential equation, {\it Appl. Math. Comp.}, {\bf 266} (2015), 441-461.

\bibitem{zwh15} Zong, X.F., Wu F.K., Huang C.M., Exponential mean square stability of the theta approximations for neutral stochastic differential delay equations, {\it J. Comp. Appl. Math.}, {\bf 286} (2015), 172-185.

\bibitem{zw16} Zong, X.F., Wu F.K., Exponential stability of the exact and numerical solutions for neutral stochastic delay differential equations, {\it Appl. Math. Model.}, {\bf 40} (2016), 19-30.


\end{thebibliography}
\end{document}